\documentclass[10pt,leqno]{article}
 
\usepackage{amsmath,amssymb,amsthm,mathrsfs,dsfont,bm}

\usepackage[margin=3cm]{geometry} 

\usepackage{titlesec,hyperref}

\usepackage{color}

\usepackage{fancyhdr}
\titleformat{\subsection}{\it}{\thesubsection.\enspace}{1pt}{}

\newtheorem{theo}{Theorem}[section]
\newtheorem{lemm}[theo]{Lemma}
\newtheorem{defi}[theo]{Definition}

\newtheorem{prop}[theo]{Proposition}
\newtheorem{rema}[theo]{Remark}
\numberwithin{equation}{section}

\allowdisplaybreaks 

\begin{document}
\title{Global regularity and optimal decay estimates of
	large solutions to the compressible FENE system
	\hspace{-4mm}
}

\author{Zhaonan $\mbox{Luo}^1$ \footnote{Email: luozhn@fudan.edu.cn},\quad
	Zhiying $\mbox{Meng}^2$\footnote{E-mail:  47-368@gduf.edu.cn} \quad and\quad
	Zhaoyang $\mbox{Yin}^{3}$\footnote{E-mail: mcsyzy@mail.sysu.edu.cn}\\
$^1\mbox{School}$ of Mathematical Sciences, Fudan University, Shanghai 200433, China.\\
$^2\mbox{Department}$ of Applied Mathematics, Guangdong University of Finance, \\
Guangzhou, 510521, PR China\\
$^3\mbox{School}$ of Science, Shenzhen Campus of Sun Yat-sen University, \\Shenzhen 518107, China}

\date{}
\maketitle
\hrule

\begin{abstract}
	In this paper, we are concerned with the compressible FENE dumbbell model. By virtue of the dissipative structure and the interpolation method, we firstly prove global regularity in $H^2$ framework for the compressible FENE system with some large data. Then, we obtain optimal
	decay estimates of large solutions in $H^1$ and remove the smallness assumption of low frequencies by virtue
	of the Fourier splitting method and the Littlewood-Paley decomposition theory. Furthermore, we
	establish optimal decay rate for the highest derivative of the solutions by a different method combining
	time frequency decomposition and the time weighted energy estimate.
	These obtained results generalize and cover the classical results
	of the incompressible FENE dumbbell model. 
	
\vspace*{5pt}
	
	\vspace*{5pt}
\end{abstract}

\vspace*{10pt}

\tableofcontents

	\section{Introduction}
	
	$~~$The paper is devoted to the study of the following compressible FENE system \cite{Bird1977,LZNcom.co.small.2021Global}:
		\begin{align}\label{eq0}
		\left\{
		\begin{array}{ll}
			\varrho_t+{\rm div}(\varrho u)=0 , \\[1ex]
			(\varrho u)_t+{\rm div}(\varrho u\otimes u)-{\rm div}\Sigma{(u)}+ \nabla{P(\varrho)}= {\rm div}\tau, \\[1ex]
			\psi_t+u\cdot\nabla\psi={\rm div}_{R}[- \sigma(u)\cdot{R}\psi+\nabla_{R}\psi+\nabla_{R}\mathcal{U}\psi],  \\[1ex]
			\tau_{ij}=\int_{B}(R_{i}\nabla_{Rj}\mathcal{U})\psi dR, \\[1ex]
			\varrho|_{t=0}=\varrho_0,~~u|_{t=0}=u_0,~~\psi|_{t=0}=\psi_0, \\[1ex]
			(\nabla_{R}\psi+\nabla_{R}\mathcal{U}\psi)\cdot{n}=0 ~~~~ \text{on} ~~~~ \partial B(0,R_{0}).\\[1ex]
		\end{array}
		\right.
	\end{align}
In \eqref{eq0}, the unknown of this model $(\varrho,u)=(\varrho,u)(t,x),\psi=\psi(t,x,R)$ denote respectively  the density of the solvent, velocity of the polymeric liquid and distribution function of the internal configuration, where $x\in \mathbb{R}^d,d=\{2,3\}$ and $R\in B \triangleq B(0,R_0)$ denotes the dumbbell extension.
Here the stress tensor $\Sigma(u)$, the pressure $P(\varrho),$ the potential of dumbbell force $\mathcal{U}$  and the additional stress tensor $\sigma(u)$ are given by 
\begin{align*}
	&\Sigma(u)=\mu(\nabla u+\nabla^{\mathrm{T}} u)+\mu'{\rm div}u\cdot Id,\\
	&P(\varrho)=a\varrho^{\gamma}~  with ~\gamma\geq 1,a>0,\\
	&\mathcal{U}(R)=-k\log(1-(\frac{|R|}{|R_{0}|})^{2})~ for~k>0,\\
	 & \sigma(u)=\frac{\nabla u-(\nabla u)^{\mathrm{T}}}{2}~(co\text{-}rotational~case)~or~ \sigma(u)=\nabla u ~(general~case),
\end{align*}
with the viscosity coefficient $\mu,\mu'$  satisfy the condition $2\mu+\mu'>0,\mu>0.$
 For notational simplicity, we take $a=\gamma=R_{0}=1$ in this article. For a detailed introduction of physical and mechanical background for \eqref{eq0}, one  can refer to \cite{Bird1977,Doi1988}.

	The system \eqref{eq0} has a trivial solution $(\varrho,u,\psi)$ with $(\varrho,u)=(1,0)$ and $$\psi_{\infty}(R)=\frac{e^{-\mathcal{U}(R)}}{\int_{B}e^{-\mathcal{U}(R)}dR}=\frac{(1-|R|^2)^k}{\int_{B}(1-|R|^2)^kdR}.$$  
Therefore, the purpose of this article is the study of the perturbations near the global equilibrium. Let
$(\rho,u,g)=(\varrho-1,u,\frac {\psi-\psi_\infty} {\psi_\infty})$ and $h(\rho)=\frac{\rho}{\rho+1}.$

In general case, the system \eqref{eq0} is now reduced to the following system for $(\rho,u,g)$:
\begin{align}\label{eq1}
\left\{
\begin{array}{ll}
	\rho_t+{\rm div}u=F, \\[1ex]
	u_t+\nabla \rho-{\rm div}\Sigma(u)-{\rm div}\tau=H, \\[1ex]
	g_t+\mathcal{L}g+{\rm div}u-\nabla uR\nabla_R\mathcal{U}=G,  \\[1ex]
	\tau_{ij}(g)=\int_{B}(R_{i}\nabla_{Rj}\mathcal{U})g\psi_\infty dR, \\[1ex]
	\rho|_{t=0}=\rho_0,~~u|_{t=0}=u_0,~~g|_{t=0}=g_0, \\[1ex]
	\psi_\infty\nabla_{R}g\cdot{n}=0 ~~~~ \text{on} ~~~~ \partial B(0,1) ,\\[1ex]
\end{array}
\right.
\end{align}
where  
\begin{align*}
	&F=-{\rm div}(\rho u),~~ H=-u\cdot\nabla u-h(\rho) {\rm div}\Sigma{(u)}+h(\rho)\nabla \rho-h(\rho) {\rm div}\tau, \\
	&G=-u\cdot\nabla g-\frac 1 {\psi_\infty}\nabla_R\cdot(\sigma(u) Rg\psi_\infty),~~\mathcal{L}g=-\frac 1 {\psi_\infty}\nabla_R\cdot(\psi_\infty\nabla_{R}g).
\end{align*}

		\textbf{(1) Short reviews for the incompressible FENE model}
	
	With potential $\mathcal{U}(R)=(1-|R|^2)^{1-\sigma}$ for $\sigma>1$, the first local well-posedness result in Sobolev spaces was established by  M. Renardy \cite{Renardy}. For the co-rotation case in two dimensions, F. Lin, P. Zhang, and Z. Zhang \cite{F.Lin} obtained a global existence results $k > 6$. In \cite{Masmoudi2008}, N. Masmoudi proved the global well-posedness of the  general FENE model near equilibrium with $k>0.$ 
	In the co-rotation case with $d=2$, the author \cite{Masmoudi2008} studied a global result for $k>0$ without any small conditions.  However, the global existence result with large initial data to the
	general FENE dumbbell model is still open (see \cite{InF2020,Masmoudi2016}). In \cite{Masmoudi2013}, N. Masmoudi established the global existence of weak solutions in $L^2$ under some entropy conditions.

	In a recent work, the $L^2$ decay of the velocity for the co-rotation
	FENE dumbbell model was proved by M. Schonbek \cite{Schonbek} and she established the
	decay rate $(1+t)^{-\frac{d}{4}+\frac{1}{2}}$, $d\geq 2$ with initial data  $u_0\in L^1$.
	On the other hand, she gave a conjecture with the sharp decay rate should be $(1+t)^{-\frac{d}{4}}$,~$d\geq 2$. By reason of the additional stress tensor, see \cite{Schonbek1985}, she could not verify the conjecture by using the bootstrap
	argument. Based on  M. Schonbek's work,   W. Luo and Z. Yin \cite{Luo-Yin,Luo-Yin2} improved Schonbek's result
	and showed that the decay rate is $(1+t)^{-\frac{d}{4}}$ with $d\geq 2$.
	This result showed that M. Schonbek's conjecture is true when $d \geq 2$. The optimal time decay rate  was proved by  \cite{Jia-Ye}  for a different method.
	
	\textbf{(2) Short reviews for the compressible Navier-Stokes (CNS) equations}
	
	When $\psi\equiv 0$, the system becomes the CNS equations. Now, we present some results about the CNS equations. J. Nash firstly proved the local existence and uniqueness \cite{miaocompressns23} for smooth initial data without vacuum. The global well-posedness near equilibrium in three dimension was established by A. Matsumura and T. Nishida \cite{Matsumura} for smooth data. Scaling invariance plays a fundamental role in the study of the CNS equations. R. Danchin \cite{miaocompressns11,miaocompressns155}  established a lot of global existence and uniqueness results in the "critical spaces".

	On the other hand, the long time behaviour of solutions for the CNS equations has
	many mathematical results. A. Matsumura and T. Nishida obtained the large time behaviour of the global solutions $(\rho,u)$ of the $3d$ CNS equations, see \cite{Matsumura}. 
	Recently, for $d=3$, H. Li and T. Zhang \cite{Li2011Large} proved the optimal time decay rate, which depends on  spectrum analysis in Sobolev spaces. For $d=2,$ the large time behaviour in the critical Besov space  was presented recently by R. Danchin  and J. Xu \cite{2016Optimal}. J. Xu \cite{Xu2019} obtained the optimal time decay rate in some Besov spaces with negative index, which requires the smallness of low-frequency. Very recently, Z. Xin  and J. Xu \cite{2018Optimal} studied the large time behaviour without the smallness restriction of low frequencies.
	
	\textbf{(3) The compressible FENE system}
	
	For the co-rotation case, Z. Luo, W. Luo  and Z. Yin \cite{LZNcom.co.small.2021Global} obtained the global existence ($d\geq 2$) and the $L^2$ optimal decay rate of strong  solutions, which requires small initial data   $(\rho_0,u_0,g_0)\in H^s\times H^s\times H^s(\mathcal{L}^2), s>1+\frac d2$. On the other hand, 
	the authors \cite{LZN.com.2021Global} also studied the optimal $H^1$ decay rate of global strong solutions
	for the general compressible FENE model  with small data and $d\geq 2$. 
	It is worth noting that the condition in \cite{LZN.com.2021Global} of the optimal  $L^2$ decay rate of solutions is weaker than the condition in \cite{LZNcom.co.small.2021Global}. 
	
	To our best knowledge, global critical regularity and optimal decay estimates for the compressible FENE system with large initial data have not been studied yet. This problem is interesting and more difficult than the case with small data. In this paper, we are devoting to the study of  the global regularity and optimal decay estimates of the compressible FENE system with some large data in the  framework of $H^2(\mathbb{R}^d),d\in\{2,3\}$. For the  case $d=2,$ 
	the key point is to prove a new global priori estimate for \eqref{eq1} with some large data. Using  interpolation method and the cancellation relation between the CNS equations and Fokker-Planck equation, we obtain a global priori estimate. To obtain optimal time decay rate in $H^2$, we give a more precise higher order derivatives estimate in the proof of global existence for strong solutions to \eqref{eq1}. Similarly, for the case  $d=3,$ we can  obtain the global existence results of the solutions (see Remark \ref{Re1}) and the optimal decay rate of each derivative of the solutions in $L^2$ 
	(see Remark \ref{gg}). 
		
	For the critical case $d=2$, optimal decay estimates for the compressible FENE system \eqref{eq1} with large data has not been studied yet. We now develop a universal method to obtain optimal $H^1$ decay rate with large data. Firstly, by using the method in \cite{Luo-Yin}, we  prove logarithmic decay rate of solutions. The main obstacle  is obtaining the initial polynomial decay rate. By virtue of the time weighted energy estimate and logarithmic decay rate, we improve initial polynomial decay rate for the solutions in $H^1$. This new Fourier splitting method plays a key role in the proof of optimal time decay rate. By using the time decay rate and the estimate of the stress tensor $\tau$ with $k>0$, we obtain $\|u\|_{L^\infty(0,\infty; \dot{B}^{-1}_{2,\infty})}\leq C$. Finally, we obtain optimal time decay estimates for the solutions in $H^1$. Furthermore, we establish optimal decay rate  of the solutions in $\dot{H}^2$ by a different method combining
	time frequency decomposition and the time weighted energy estimate. On the other hand, by applying the optimal decay rate of $\|\nabla u\|_{H^1}$, we  improve the decay rate of $\|\nabla^s g\|_{L^2(\mathcal{L}^2)}$ to $(1+t)^{-\frac{2+s}{2}}$ with $s\in [0,1].$ For the case $d=3,$ the proof of  optimal decay rate of solutions with large data is simpler, so we omit it.
		
	Compared with previous work, the new ingredients in
	our paper are as follows:
	
	$({\rm i})$ Expanding the range of initial value corresponding to the global solutions. Compared with the co-rotation incompressible FENE model, the difficulty of the  compressible FENE model  is the lack of conservation law. By using  the dissipative structure of $(\rho,u,g)$, the continuity argument and the special interpolation method, we obtain global existence of the  compressible FENE system. Let $d=\{2,3\}$ and the critical regularity with $\beta=2.$ Then the initial data satisfy the following conditions for $\epsilon$ small enough
	\begin{align*}
		E_1(0)+E_2(0)\leq M,~~E_0(0)+E_1(0)\leq \epsilon(M),
	\end{align*} 
	which implies that the $H^2$ norm of initial data is large. It is worth mentioning that  we obtain the first result of  global existence of the compressible FENE model with large initial value. Moreover, we get the following two closed energy estimates:
	\begin{align*} 
		&E_0(t)+E_1(t)+\eta(u,\nabla\rho)+\int_0^tD_0(t)+D_1(t)+2\eta\|\nabla\rho\|^2_{L^2}dt'\leq C (E_0(0)+E_1(0)),\\
		&E_2(t)+\eta(\nabla u,\nabla^2\rho)+\int_0^tD_2(t)+2\eta\|\nabla^2\rho\|^2_{L^2}dt'\leq C (E_1(0)+E_2(0)).
	\end{align*}
	
	$({\rm ii})$ Optimal decay rate for the highest derivative of the solutions. Owing to $E_2+\eta(\nabla u,\nabla^2 \rho)\ncong E_2$ with $\eta$  small enough, we can not obtain decay rate
	of $(\rho,u)$ in $\dot{H}^2$  by virtue of the same approach as in $H^1.$ 
	By the improved Fourier splitting method,  the time weighted energy estimate and the optimal decay rate of $(\rho,u)$ in $H^1$, we prove the optimal decay rate for  the highest derivative of  $(\rho,u)$ in $L^2.$ In addition, one also see that  the decay rate of $\nabla^{\beta}g$ in $L^2(\mathcal{L}^2)$ with $\beta\in[0,2]$ is the same as that of $\nabla^{\beta}\nabla u.$ 
	 Note that the result of optimal decay rate for the highest derivative is novel.

	The article has the following structure. In Section \ref{sec2}, we give some notations and main results. In Section \ref{sec3}, we are  devoted to the study of  global regularity of solutions to \eqref{eq0} and $H^2$ decay estimates for the compressible FENE dumbbell system with some large data.
	
	\section{\textbf{Notations and main results}}\label{sec2}
	
	$~~$In this section we will introduce some notations that will be used in the
sequel and present our main results as follows. 

	\subsection{\textbf{Notations and functional spaces}} 
We give the following notations and  some functional spaces.
	
	$(1)~\nabla u ({\rm div}u)$ represents the gradient(divergence) of $u$ with respect to the variable $x;$

	 $(2)~\hat{f}=\mathcal{F}(f)$ is the Fourier transform of $f$ and $\Lambda^sf=\mathcal{F}^{-1}(|\xi|^s \widehat{f} );$ 
	
	$(3)~	\mathcal{L}^{p}=\big\{f \big|\|f\|^{p}_{\mathcal{L}^{p}}=\int_{B} \psi_{\infty}|f|^{p}dR<\infty\big\};$
	
	$(4)~L^{p}(\mathcal{L}^{q})=L^{p}[\mathbb{R}^{d};\mathcal{L}^{q}]=\big\{f \big|\|f\|_{L^{p}_{x}(\mathcal{L}^{q})}=(\int_{\mathbb{R}^{d}}(\int_{B} \psi_{\infty}|f|^{q}dR)^{\frac{p}{q}}dx)^{\frac{1}{p}}<\infty\big\};$
	
	$(5)~H^{s}(\mathcal{L}^{2})=\{f\big| \|f\|^2_{H^{s}(\mathcal{L}^{2})}=\int_{\mathbb{R}^{d}}\int_B(|f|^2+|\Lambda^s f|^2)\psi_\infty dRdx<\infty\};$	
	
	$(6)~$The energy and energy dissipation functionals for $(\rho,u,g)$ as follows:
	\begin{align*}
			&E_{\beta}(t)=\|\nabla^{\beta}(\rho,u)\|^2_{L^2}+\|\nabla^{\beta}g\|^2_{L^2(\mathcal{L}^2)},\\
			&D_{\beta}(t)=\mu\|\nabla^{\beta}\nabla u\|^2_{L^2}+(\mu+\mu')\|\nabla^{\beta}{\rm div} u\|^2_{L^2}+\|\nabla^{\beta}\nabla_R g\|^2_{L^2(\mathcal{L}^2)},\\
			&\mathcal{E}_{\delta}(t)=\|\nabla^{\delta}(\rho,u)\|^2_{H^{2-\delta}}+\|\nabla^{\delta}g\|^2_{H^{2-\delta}(\mathcal{L}^2)}+2\eta(\nabla^{\delta}u,\nabla\nabla^{\delta}\rho)_{H^{1-\delta}},\\
			&\mathcal{D}_{\delta}(t)={\mu}\|\nabla^{\delta}\nabla u\|^2_{H^{2-\delta}}+(\mu+\mu')\|\nabla^{\delta}{\rm div}u\|^2_{H^{2-\delta}}+\|\nabla^{\delta}\nabla_Rg\|^2_{H^{2-\delta}(\mathcal{L}^2)}+\frac{\eta}{2}\|\nabla\nabla^{\delta}\rho\|^2_{H^{1-\delta}}.
		    	\end{align*}
	where ${\beta}\in[0,2],\delta\in[0,1]$ and $\eta$  small enough. 
Next we recall the Littlewood-Paley decomposition.

		Let $\mathcal{C}$ be the annulus $\{\xi\in\mathbb{R}^d:\frac 3 4\leq|\xi|\leq\frac 8 3\}$. There exist radial functions $\varphi$, valued in the interval $[0,1]$, belonging respectively to $\mathcal{D}(\mathcal{C})$, and such that
		$$ \forall\xi\in\mathbb{R}^d\backslash\{0\},\ \sum_{j\in\mathbb{Z}}\varphi(2^{-j}\xi)=1, $$
		$$ |j-j'|\geq 2\Rightarrow\mathrm{Supp}\ \varphi(2^{-j}\cdot)\cap \mathrm{Supp}\ \varphi(2^{-j'}\cdot)=\emptyset. $$
		Furthermore, we have
		$$ \forall\xi\in\mathbb{R}^d\backslash\{0\},\ \frac 1 2\leq\sum_{j\in\mathbb{Z}}\varphi^2(2^{-j}\xi)\leq 1. $$

	Let $u$ be a tempered distribution in $\mathcal{S}'_h(\mathbb{R}^d)$. For all $j\in\mathbb{Z}$, define
	$$
	\dot{\Delta}_j u=\mathcal{F}^{-1}(\varphi(2^{-j}\cdot)\mathcal{F}u).
	$$
	Then the Littlewood-Paley decomposition is given as follows:
	$$ u=\sum_{j\in\mathbb{Z}}\dot{\Delta}_j u \quad \text{in}\ \mathcal{S}'(\mathbb{R}^d). $$
	Now, we introduce the definition of homogenous Besov spaces as follows.
	\begin{defi}
	Let $s\in\mathbb{R},\ 1\leq p,r\leq\infty$ and $1\leq q\leq p\leq \infty.$ The homogeneous Besov space $\dot{B}^s_{p,r}$ and $\dot{B}^s_{p,r}(\mathcal{L}^q)$ are defined by
	$$ \dot{B}^s_{p,r}=\{u\in \mathcal{S}'_h:\|u\|_{\dot{B}^s_{p,r}}=\Big\|(2^{js}\|\dot{\Delta}_j u\|_{L^p})_j \Big\|_{l^r(\mathbb{Z})}<\infty\}, $$
	$$ \dot{B}^s_{p,r}(\mathcal{L}^q)=\{\phi\in \mathcal{S}'_h:\|\phi\|_{\dot{B}^s_{p,r}(\mathcal{L}^q)}=\Big\|(2^{js}\|\dot{\Delta}_j \phi\|_{L_{x}^{p}(\mathcal{L}^q)})_j \Big\|_{l^r(\mathbb{Z})}<\infty\}.$$
\end{defi}
\subsection{\textbf{Main results}}
In this section, we state our main results.      
\begin{theo}(Global regularity)\label{global solution'}
	Let $d=2.$ Assume that the initial data $(\rho_0,u_0,g_0)\in H^2\times H^2\times H^2(\mathcal{L}^2)$ and satisfy the conditions $\int_{B}g_0\psi_{\infty}dR=0$ and $1+g_0>0.$  For any constant $M$, there  exists a constant $\epsilon$  ($\epsilon$ is small enough and depends on $M$) such that  if 
	\begin{align}\label{initial}
		E_1(0)+E_2(0)\leq M,	~~~and~~~E_0(0)+E_1(0)\leq \epsilon,
	\end{align}
	then the system \eqref{eq1} admits a unique global solution $(\rho,u,g) $ with $\int_B g\psi_{\infty}dR=0$ and $1+g>0.$ 
	
	Moreover, we have 
	\begin{align*}
		\sup_{t\in [0,\infty)} \mathcal{E}_0(t)+\int_0^{\infty} \mathcal{D}_0(t)dt\leq C\mathcal{E}_0(0),
	\end{align*}
	where $C>1$ is a constant.
\end{theo}
\begin{rema}
	Taking $\rho=1,$ the system \eqref{eq1} reduces to the incompressible FENE dumbbell model. These results of the global regularity of strong solutions in the current manuscript improve the result of  \cite{Masmoudi2008}.
\end{rema}
\begin{rema}\label{Re1}
	Under the conditions that the initial value satisfies Theorem \ref{global solution'}, the above result of global regularity is also valid for the  compressible co-rotational FENE model \eqref{eq0} with $\sigma(u)=\frac{\nabla u-(\nabla u)^{\mathrm{T}}}{2}$ and  the compressible FENE model with $d=3$. For the case $d=3,$ the proof of global existence is similar to $d=2,$ then we omit it in this paper. 
\end{rema}
\begin{theo} (Optimal decay estimates) \label{opdecay'}
	Let  $(\rho_0,u_0,g_0)$ satisfy the conditions in Theorem \ref{global solution'}. Assume further that $(\rho_0,u_0,g_0)\in \dot{B}^{-1}_{2,\infty}\times\dot{B}^{-1}_{2,\infty}\times\dot{B}^{-1}_{2,\infty}(\mathcal{L}^2),$ then it holds that for $s\in [0,1]$, 
	\begin{align*}
		\|\nabla^{s}\rho\|_{L^2}+\|\nabla^{s}u\|_{L^2}\leq C(1+t)^{-\frac {1+s} {2}},~~\|\nabla^{s}g\|_{L^2(\mathcal{L}^2)}\leq C(1+t)^{-\frac {2+s} {2}},
	\end{align*}
and
\begin{align*}
	\|\nabla^{2}\rho\|_{L^2}+\|\nabla^{2}u\|_{L^2}+\|\nabla^{2}g\|_{L^2(\mathcal{L}^2)}\leq C(1+t)^{-\frac 3 2}.
\end{align*}
\end{theo}

\begin{rema}\label{Re2} initial condition
	Under the conditions that the initial value satisfies Theorem \ref{opdecay'},	the above results of optimal decay rate for $(\rho,u)$ is also valid for the compressible co-rotational FENE model and the optimal decay rate of $g$ satisfies $\|g\|_{H^2(\mathcal{L}^2)}\leq Ce^{-ct}$. In addition, if $(\rho_0,u_0,g_0)\in H^3\times H^3\times H^3(\mathcal{L}^2)$ with $d=2,$  the optimal decay rate of $\|\nabla^2g\|_{L^2(\mathcal{L}^2)}$ is $(1+t)^{-2}.$ These results of the optimal decay estimates of
	strong solutions in the current manuscript is an extension of \cite{Luo-Yin2,luo2020large} to the compressible fluid. 
\end{rema}
\begin{rema}\label{gg}
	For $d=3,$ 	let the initial value satisfy the conditions in Theorem \ref{global solution'}. Assume further that $(\rho_0,u_0,g_0)\in \dot{B}^{-\frac 32}_{2,\infty}\times\dot{B}^{-\frac 32}_{2,\infty}\times\dot{B}^{-\frac 32}_{2,\infty}(\mathcal{L}^2).$ Then, for any $\beta\in [0,2],$  the optimal decay rate of $\|\nabla^{\beta}(\rho,u)\|_{L^2}$ is $(1+t)^{-\frac{3+2\beta}{4}}$ and  the optimal decay rate of $\|\nabla^{s} g\|_{L^2(\mathcal{L}^2)}$ is $(1+t)^{-\frac{5+2s}{4}}$ with $s\in [0,1].$ In addition, if $(\rho_0,u_0,g_0)\in H^3\times H^3\times H^3(\mathcal{L}^2)$, the optimal decay rate of $\|\nabla^2g\|_{L^2(\mathcal{L}^2)}$ is $(1+t)^{-\frac 94}.$ Note that the proof of optimal decay rate of solutions is simpler. Therefore, we omit the  proof of the optimal decay  rate for $d=3.$
\end{rema}
\begin{rema}
Together with the result in \cite{Li2011Large,Xu2019}, one can see that the $H^2$ decay estimates of $(\rho,u)$  is optimal in Theorem \ref{opdecay'}.
\end{rema}

\section{Global regularity and optimal decay estimates with large data}\label{sec3}
We start the proof of main results by the following  lemmas.
\subsection{\textbf{Interpolation inequality and Hardy type inequality}}
     \begin{lemm}\cite{1959On}\label{intef}
		Let $d\geq2,~p\in[2,+\infty)$ and $0\leq s,s_1\leq s_2$, then there exists a constant $C$ such that
		$$\|\Lambda^{s}f\|_{L^{p}}\leq C \|\Lambda^{s_1}f\|^{1-\theta}_{L^{2}}\|\Lambda^{s_2} f\|^{\theta}_{L^{2}},$$
		where $0\leq\theta\leq1$ and $\theta$ satisfy
		$$ s+d(\frac 1 2 -\frac 1 p)=s_1 (1-\theta)+\theta s_2.$$
		Note that we require that $0<\theta<1$, $0\leq s_1\leq s$, when $p=\infty$.
	\end{lemm}
\begin{lemm}\cite{Masmoudi2008}\label{Lemma1}
	If $\int_B g\psi_\infty dR=0$, then there exists a constant $C$ such that
	$$\|g\|_{\mathcal{L}^{2}}\leq C \|\nabla _{R} g\|_{\mathcal{L}^{2}}.$$
\end{lemm}
\begin{lemm}\cite{LZN.com.2021Global,Masmoudi2008}\label{luoz}
If $\int_B g\psi_{\infty}dR=0$ and $\|g\|_{\mathcal{L}^2}\leq C\|\nabla_Rg\|_{\mathcal{L}^2}$,  then  there exists a constant $C$ such that for any $k>0$,
		\begin{align*}
			|\tau(g)|\leq C\|g\|^{\frac 1 2}_{\mathcal{L}^{2}}
			\|\nabla _{R}g\|^{\frac 1 2}_{\mathcal{L}^{2}}.
		\end{align*} 
\end{lemm}
	\subsection{\textbf{Local existence and uniqueness}}
	\par
	The main part of this subsection is the proof of existence and uniqueness of local solutions. For notational simplicity,  we will adopt the following notations. 
	\begin{align*}
		&E^n_{\beta}(t)=\|\nabla^{\beta} (\rho^n, u^n)\|^2_{L^2}+\|\nabla^{\beta} g^n\|^2_{L^2(\mathcal{L}^2)}, \\
		&D^n_{\beta}(t)=\mu\|\nabla^{{\beta}+1} u^n\|^2_{L^2}+(\mu+\mu')\|\nabla^{{\beta}} {\rm div}u^n\|^2_{L^2}+\|\nabla^{\beta} \nabla_Rg^n\|^2_{L^2(\mathcal{L}^2)},
	\end{align*}
with ${\beta}\in[0,2].$
	\begin{prop}(Local existence)\label{local existence }
	Let $d=2.$ Let $(\rho_0,u_0,g_0)\in H^2\times H^2\times H^2(\mathcal{L}^2)$ and  satisfy the conditions $\int_{B}g_0\psi_{\infty}dR=0$ and $1+g_0>0.$ For any  constant $M,$ there exists a constant $\epsilon(M)$ such that
	\begin{align*}
		(E_1+E_2)(0)\leq M,~~~(E_0+E_1)(0)\leq \epsilon,
	\end{align*}
where $\epsilon$ small enough.

	Then, the system \eqref{eq1} has a unique local solution $$(\rho,u,g)\in  L^{\infty}(0,T; H^2)\times L^{\infty}(0,T;H^2)\times L^{\infty}(0,T;H^2(\mathcal{L}^2)),$$  with $\int_B g\psi_{\infty}dR=0$ and $1+g>0.$
	
	 Moreover, we have
	\begin{align*}
		\sup\limits_{t\in [0,T]}\mathcal{E}_0(t)+\int_0^T \mathcal{D}_0(t)dt\leq C\mathcal{E}_0(0),
	\end{align*}
with $C>0.$
\end{prop}
		\begin{proof}
			\textbf{Step 1: Constructing approximate solutions}
			
			Starting from $(\rho^{0}(t,x), u^{0}(t,x), g^{0}(t,x,R))=(0, 0, 0)$, we define a sequence $(\rho^n, u^n, g^n)$ of smooth solutions  by solving the following equation:
			\begin{align}\label{approximate}
				\left\{
				\begin{array}{ll}
					\partial_t\rho^{n+1}+{\rm div}u^{n+1}=-\rho^{n}{\rm div}u^{n+1}-u^{n}\cdot\nabla\rho^{n+1}\triangleq F^{n+1} , \\[1ex]
					\partial_tu^{n+1}+\nabla\rho^{n+1}-{\rm div}\Sigma(u^{n+1})-{\rm div}\tau^{n+1}=-u^{n}\cdot\nabla u^{n+1}-h(\rho^n){\rm div}\Sigma{(u^{n+1})}\\[1ex]
					~~~~~~~~~+h(\rho^n) \nabla\rho^{n+1}-h(\rho^n){\rm div}\tau^{n+1}\triangleq H^{n+1}, \\[1ex]
					\partial_tg^{n+1}+\mathcal{L}g^{n+1}+{\rm div}u^{n+1}-\nabla u^{n+1}R\nabla_R\mathcal{U}=-u^{n}\cdot\nabla g^{n+1}\\[1ex]
					~~~~~~~~~-\frac 1 {\psi_\infty}\nabla_R\cdot(\sigma (u^{n})Rg^{n+1}\psi_\infty)\triangleq G^{n+1},  \\[1ex]
					(\rho^{n+1}, u^{n+1}, g^{n+1})|_{t=0}=(\rho_0(x), u_0(x), g_0(x,R)).
				\end{array}
				\right.
			\end{align}
		
	\textbf{Step 2: $(\rho^n, u^n, g^n)$ uniformly bounded in $H^2\times H^2\times H^2(\mathcal{L}^2)$}
	
	For any constant $M>0$, there exists a constant $\epsilon(M)>0$ and $T>0$ such that for any $n\in N$, we have $$(E^n_1+E^n_2)(0)\leq 2M,~(E^n_0+E^n_1)(0)\leq 2\epsilon.$$
	Assume further that
	\begin{align}
			&\sup_{t\in[0,T]}({E}^{n}_1+{E}^{n}_2)(t)+\int_0^T(D^{n}_1+D^{n}_2)(t)dt\leq 2M,\label{initial1'}\\
			&\sup_{t\in[0,T]}(E^{n}_0+E^{n}_1)(t)+\int_0^T(D^{n}_0+D^{n}_1)(t)dt\leq 2\epsilon\label{initial1''}.
	\end{align}
\noindent
Then, we will prove the following estimates
\begin{align}
		&\sup_{t\in[0,T]}(E^{n+1}_0+E^{n+1}_1)(t)+\int_0^T(D^{n+1}_0+D^{n+1}_1)(t)dt\leq 2\epsilon,\label{ee0}\\
&\sup_{t\in[0,T]}(E^{n+1}_1+E^{n+1}_2)(t)+\int_0^T(D^{n+1}_1+D^{n+1}_2)(t)dt\leq 2M.\label{ee1}
\end{align}
To begin with, for any $\beta\in \{0,1,2\},$
by using the fact that $\int_{B}g^{n+1}\psi_{\infty}dR=\int_{B}g_0\psi_{\infty}dR=0$ and integrating by parts, we deduce that
\begin{align}\label{e1}
	&(\nabla^\beta{\rm div}u^{n+1},\nabla^\beta\rho^{n+1})+(\nabla^\beta\nabla\rho^{n+1},\nabla^\beta u^{n+1})=0,\\\label{e2}
	&(\nabla^\beta{\rm div}\tau^{n+1},\nabla^\beta u^{n+1})+\left\langle\nabla^{\beta+1} u^{n+1}R\nabla_R\mathcal{U},\nabla^\beta g^{n+1}\right\rangle\\ \notag
	&=(\nabla^\beta{\rm div}\tau^{n+1},\nabla^\beta u^{n+1})+(\nabla^{\beta+1}u^{n+1},\nabla^\beta\tau^{n+1})=0,
\end{align}
and
\begin{align}\label{e3}
	&\left\langle{\nabla^\beta\rm div}~u^{n+1},\nabla^\beta g^{n+1}\right\rangle=(-1)^{\beta}\int_{\mathbb{R}^2}\nabla^{2\beta}{\rm div}~u^{n+1}\int_{B}g^{n+1}\psi_{\infty}dRdx=0.
\end{align} 
We now give the following estimates.

\textbf{({\rm i}) The energy estimate of ${E}^{n+1}_0(t)$}
 
 Multiplying $\eqref{approximate}_1,\eqref{approximate}_2,\eqref{approximate}_3$ by $\rho^{n+1},u^{n+1}$ and $ g^{n+1}$, respectively, and  integrating the
 resulting equations. 
 Then, using \eqref{e1}-\eqref{e3}, we get
\begin{align*}
	\frac1 2\frac{d}{dt}{E}^{n+1}_0(t)+{D}^{n+1}_0(t)= (F^{n+1},\rho^{n+1})+(H^{n+1},u^{n+1})+\left\langle G^{n+1},g^{n+1}\right\rangle,
\end{align*}
where \begin{align*}
&(F^{n+1},\rho^{n+1})=-(u^n\cdot \nabla \rho^{n+1},\rho^{n+1})-(\rho^n{\rm div}u^{n+1},\rho^{n+1})\triangleq 	\sum_{i=1}^{2}I_i,\\
&(H^{n+1}
		,u^{n+1})=-(u^n\cdot\nabla u^{n+1},u^{n+1})+(h(\rho^n)\nabla\rho^{n+1},u^{n+1})+(h(\rho^n)\Sigma(u^{n+1}),\nabla u^{n+1})\\ \notag
&~~~~~~~~~~~~~~~~~~~~~+(\nabla(h(\rho^n))\Sigma(u^{n+1}),u^{n+1})-(h(\rho^n){\rm div}\tau^{n+1},u^{n+1})\triangleq\sum_{i=3}^{7}I_i,\\
&\left\langle G^{n+1},g^{n+1}\right\rangle=-\left\langle u^n\cdot\nabla g^{n+1},g^{n+1}\right\rangle+\left\langle\sigma(u^n)Rg^{n+1},\nabla_Rg^{n+1}\right\rangle\triangleq \sum_{i=8}^{9}I_i	.
\end{align*}
Using Lemmas \ref{intef}-\ref{luoz}, one has 
\begin{align}\label{ineq0}
	&\|f^n\|_{L^{\infty}}\leq C\|f^n\|^{\frac  12}_{L^2}\|\nabla^2 f^n\|^{\frac  12}_{L^2},~~\|\nabla f^n\|_{L^4}\leq C\|\nabla f^n\|^{\frac  12}_{L^2}\|\nabla^2 f^n\|^{\frac  12}_{L^2},~~\\ \notag
	&\|{\rm div}\tau^n\|_{L^2}\leq C\|\nabla \nabla_Rg^n\|_{L^2(\mathcal{L}^2)},~~\|f^n\|_{L^4}\leq C\|f^n\|^{\frac12}_{L^2}\|\nabla f^n\|^{\frac12}_{L^2}.
\end{align}
Using $\eqref{ineq0}$, there holds
\begin{align*}
	&|I_1|\leq C\|u^n\|_{L^{\infty}}\|\nabla \rho^{n+1}\|_{L^2}\|\rho^{n+1}\|_{L^2}\leq C(E^n_0E^n_2)^{\frac 1 4}(E^{n+1}_0E^{n+1}_1)^{\frac 1 2},
\end{align*}
and
\begin{align*}
(I_2+I_3)+I_4+I_5\leq C(E_0^nE_2^n)^{\frac 14}\Big((E_0^{n+1}D_0^{n+1})^{\frac 12}+(E_0^{n+1}E_1^{n+1})^{\frac 12}+D_0^{n+1}\Big).
\end{align*}
It follows from $\eqref{ineq0}$ that
\begin{align*}
	&|I_6|\leq C\|\nabla\rho^n\|_{L^2}\|\nabla u^{n+1}\|_{L^4}\|u^{n+1}\|_{L^4}\leq  C(E_1^n)^{\frac 12}(E_0^{n+1})^{\frac 14}(D_0^{n+1})^{\frac 12}(D_1^{n+1})^{\frac 14},\\
	&|I_7|\leq C\|\rho^n\|_{L^{\infty}} \|{\rm div}\tau^{n+1}\|_{L^2}\|u^{n+1}\|_{L^2}\leq C(E^n_0E^n_2)^{\frac 1 4}(E^{n+1}_0D^{n+1}_1)^{\frac  12}.
\end{align*}
The terms of $I_8,I_9$ can be estimated as 
\begin{align*}
	&|I_8|\leq C(E^n_0E^n_2)^{\frac 1 4}(E^{n+1}_0D^{n+1}_0)^{\frac 1 2},~~|I_9|\leq C(E^n_1E^n_2)^{\frac 14}(E^{n+1}_0E^{n+1}_1)^{\frac 14}(D^{n+1}_0)^{\frac 12}.
\end{align*}
Gathering the estimates of $I_1-I_9$ and using the conditions of \eqref{initial1'}, \eqref{initial1''}, we arrive at the following estimate
\begin{align}\label{E0}
&\frac 1 2 \frac d {dt} E^{n+1}_0+\Big(
1-C(\epsilon^{\frac 1 4}+\epsilon^{\frac 1 2})	\Big) D^{n+1}_0\leq C(1+\epsilon^{\frac 14}+\epsilon^{\frac 12})E^{n+1}_0+C\epsilon^{\frac14}E_1^{n+1}+C(\epsilon^{\frac 1 4}+\epsilon^{\frac 12})D_1^{n+1}.
\end{align}
\textbf{({\rm ii}) The energy estimate of ${E}^{n+1}_1(t)$}

Applying $\nabla$ to \eqref{approximate} and using \eqref{e1}-\eqref{e3}, we obtain 
\begin{align*}
	\frac 12\frac{d}{dt}E^{n+1}_1(t)+D^{n+1}_{1}(t)=
	&(\nabla F^{n+1},\nabla\rho^{n+1})+(H^{n+1},\nabla^2u^{n+1}) +(\nabla G^{n+1},\nabla g^{n+1}),
\end{align*}
with \begin{align*}
(\nabla F^{n+1},\nabla\rho^{n+1})\triangleq
&\sum\limits_{i=1}^4I'_i\\ \notag=
&-(\nabla u^n\cdot\nabla\rho^{n+1},\nabla\rho^{n+1})+\frac12({\rm div}u^n,|\nabla\rho^{n+1}|^2)\\ \notag
&-(\rho^n\nabla{\rm div}u^{n+1},\nabla\rho^{n+1})-(\nabla\rho^n{\rm div }u^{n+1},\nabla\rho^{n+1}),\\
\left\langle\nabla G^{n+1},\nabla g^{n+1}\right\rangle
&\triangleq\sum\limits_{i=5}^8I'_i\\ \notag
&=-\left\langle\nabla u^n\cdot\nabla g^{n+1},\nabla g^{n+1}\right\rangle+\frac{1}{2}\left\langle{\rm div}u^n,|\nabla g^{n+1}|^2\right\rangle\\\notag
&~~~+\left\langle\sigma(u^n)R\nabla g^{n+1},\nabla\nabla_Rg^{n+1}\right\rangle+\left\langle\nabla\sigma(u^n)Rg^{n+1},\nabla\nabla_Rg^{n+1}\right\rangle.
\end{align*}
It follows from Lemma \ref{intef} that
\begin{align}\label{ineq1'}
	\|\nabla f^n\|_{L^{\infty}}\leq C\|\nabla f^n\|^{\frac 12}_{L^2}\|\nabla^3 f^n\|^{\frac 12}_{L^2},~\|\nabla^2 f^n\|_{L^4}\leq C\|\nabla f^n\|^{\frac 14}_{L^2}        \|\nabla^3 f^n\|^{\frac 34}_{L^2}.                                    
\end{align}
Using $\eqref{ineq1'}$, we  derive that
\begin{align*}
	|I'_1|+|I'_2|\leq  C\|\nabla u^n\|_{L^{\infty}}\|\nabla \rho^{n+1}\|^2_{L^2}\leq C (E^{n}_1D^n_2)^{\frac 14}E^{n+1}_1.
\end{align*}
It follows from $\eqref{ineq1'}$ that
\begin{align*}
|I'_3|\leq C(E^n_0E_2^n)^{\frac 14}(D_1^{n+1}E_1^{n+1})^{\frac 12},
\end{align*}
and
\begin{align*}
	|I'_4|\leq C\|\nabla\rho^n\|_{L^4}\|\nabla u^{n+1}\|_{L^4}\|\nabla\rho^{n+1}\|_{L^2}\leq C (E^n_1E^n_2)^{\frac 14}(D^{n+1}_0D^{n+1}_1)^{\frac 14}(E^{n+1}_1)^{\frac 12}.
\end{align*}
According to Lemma \ref{Lemma1} and $\eqref{ineq1'}$, we derive that
\begin{align*}
	|I'_5|+|I'_6|+|I'_7|\leq C\|\nabla u^n\|_{L^{\infty}}\|\nabla g^{n+1}\|_{L^2(\mathcal{L}^2)}\|\nabla \nabla_Rg^{n+1}\|_{L^2(\mathcal{L}^2)}\leq C (E^{n}_1D^n_2)^{\frac 14}(E^{n+1}_1D^{n+1}_1)^{\frac 12}.
\end{align*}
Using $\eqref{ineq1'}$ and Lemma \ref{Lemma1}, one has
\begin{align*}
		|I'_8|
		&\leq C \|\nabla^2u^n\|_{L^4}\|g^{n+1}\|_{L^4(\mathcal{L}^2)}\|\nabla\nabla_Rg^{n+1}\|_{L^2(\mathcal{L}^2)}\notag \\ 
		&\leq C (E^n_1)^{\frac 18}(D_2^n)^{\frac 38}(E^{n+1}_0E^{n+1}_1)^{\frac 14}(D^{n+1}_1)^{\frac 12}\leq C(E^n_1)^{\frac 18}\Big(D^n_2(E^{n+1}_0+E^{n+1}_1)+D^{n+1}_1\Big).
\end{align*}
Now we give the estimate  of  $(H^{n+1},\nabla^2 u^{n+1})$.  Integrating by parts and using Lemmas \ref{intef}-\ref{luoz},  one arrives at
\begin{align}\label{ineq4'}
	(H^{n+1},\nabla^2 u^{n+1})
	&\leq C\|H^{n+1}\|_{L^2}\|\nabla^2 u^{n+1}\|_{L^2}\\ \notag
	&\leq C(E^n_0E^n_2)^{\frac 14}(D^{n+1}_1)^{\frac 12}\Big((E^{n+1}_1)^{\frac 12}+(D^{n+1}_1)^{\frac 12}+(D^{n+1}_0)^{\frac 12}\Big).
\end{align}
This together with the estimates of $I'_1-I'_8$ and \eqref{initial1'}, \eqref{initial1''}, we  have 
\begin{align}\label{E1}
	&\frac 12\frac{d}{dt}E_1^{n+1}+\Big(1-C(\epsilon^{\frac 18}+\epsilon^{\frac 14})\Big)D^{n+1}_1\\ \notag
	&\leq C\epsilon^{\frac 18}D^n_2E^{n+1}_0+C\Big(\epsilon^{\frac 14}+(\epsilon^{\frac 18}+\epsilon^{\frac 14})D^n_2\Big)E^{n+1}_1+C\epsilon^{\frac 14}D^{n+1}_0.
\end{align}
Define $$A^n_0(t)\triangleq (E^n_0+E^n_1)(t)+\int_0^t(D^n_0+D^n_1)(t)dt'.$$
Combining \eqref{E0} and \eqref{E1}, we have
\begin{align*}
	&\frac 12(E^{n+1}_0+E^{n+1}_1)+\Big(1-C(\epsilon^{\frac12}+\epsilon^{\frac14}+\epsilon^{\frac18})\Big)\int_0^t(D^{n+1}_0+D^{n+1}_1)dt'\\
	&\leq \frac 12(E^{n+1}_0+E^{n+1}_1)(0)+C\Big(T(1+\epsilon^{\frac14})+(\epsilon^{\frac 12}+\epsilon^{\frac14}+\epsilon^{\frac18})\Big)\sup\limits_{t\in [0,T]}A^{n+1}_0(t).
\end{align*}
Taking $T\leq \epsilon^{\frac 12}\leq \frac{1}{32C},$  for any $t\in [0,T]$, one can get  
\begin{align}\label{EE0}
	\sup\limits_{t\in [0,T]}A^{n+1}_0(t)\leq \epsilon.
\end{align}
 \textbf{({\rm iii}) The energy estimate of ${E}^{n+1}_2(t)$}
 
Applying $\nabla^2$ to \eqref{approximate} and multiplying the resulting equation by  $(\nabla^2\rho^{n+1},\nabla^2u^{n+1},\nabla^2g^{n+1})$ and applying \eqref{e1}-\eqref{e3}, we have 
\begin{align*}
	\frac 12\frac{d}{dt}E^{n+1}_2(t)+D^{n+1}_2(t)
	&=(\nabla^2F^{n+1},\nabla^2\rho^{n+1})+(\nabla^2 H^{n+1},\nabla^2u^{n+1})+\left\langle\nabla^2G^{n+1},\nabla^2g^{n+1}\right\rangle.
\end{align*}
Integrating by parts, we get the following equalities
\begin{align*}
	(\nabla^2F^{n+1},\nabla^2\rho^{n+1})
	&=-(\rho^n\nabla^2{\rm div}u^{n+1},\nabla^2\rho^{n+1})-2(\nabla{\rm div}u^{n+1}\nabla\rho^n,\nabla^2\rho^{n+1})\\&~~~-(\nabla^2u^n\cdot\nabla\rho^{n+1},\nabla^2\rho^{n+1})-({\rm div}u^{n+1}\nabla^2\rho^n,\nabla^2\rho^{n+1})\\
	&~~~-2(\nabla^2\rho^{n+1}\nabla u^n,\nabla^2\rho^{n+1})+\frac 12({\rm div}u^n,|\nabla^2\rho^{n+1}|^2)\\ \notag	&\triangleq\sum_{i=1}^6I^{''}_i ,\\	
	(\nabla^2H^{n+1},\nabla^2u^{n+1})
	&=\big(u^n\cdot\nabla^2 u^{n+1}+h(\rho^n)\nabla^2\rho^{n+1}+h(\rho^n)\nabla{\rm div}\Sigma(u^{n+1})\\
	&~~~+h(\rho^n)\nabla {\rm div}\tau^{n+1},\nabla^3u^{n+1}\big) 
	\\
	&~~~+\big(\nabla u^n\cdot\nabla u^{n+1}-\nabla h(\rho^n)\nabla\rho^{n+1}+\nabla h(\rho^n){\rm div}\Sigma(u^{n+1})\\ \notag
	&~~~+\nabla h(\rho^n){\rm div}\tau^{n+1},\nabla^3u^{n+1}\big)\\ \notag
	&\triangleq I''_7+I''_8,
	\\
\left\langle\nabla^2 G^{n+1},\nabla^2g^{n+1}\right\rangle
	&=-\left\langle\nabla^2u^n\cdot\nabla g^{n+1},\nabla^2g^{n+1}\right\rangle+2\left\langle\nabla\sigma(u^n)R\nabla g^{n+1},\nabla^2\nabla_{R}g^{n+1}\right\rangle\\ \notag 
	&~~~-2\left\langle\nabla u^n\cdot \nabla^2g^{n+1},\nabla^2g^{n+1}\right\rangle+\frac{1}{2}\left\langle{\rm div}u^n,|\nabla^2g^{n+1}|^2\right\rangle\\ \notag &~~~-\left\langle\sigma(u^n)R\nabla^2g^{n+1},\nabla^2\nabla_Rg^{n+1}\right\rangle+\left\langle\nabla^2\sigma(u^n)Rg^{n+1},\nabla^2\nabla_Rg^{n+1}\right\rangle\\ \notag
	&\triangleq\sum\limits_{i=9}^{14}I''_{i}\\ \notag.
\end{align*}
It follows from Lemma \ref{intef} that
\begin{align*}
	&|I''_1|\leq C\|\nabla^3 u^{n+1}\|_{L^2}\|\rho^n\|_{L^\infty}\|\nabla^2\rho^{n+1}\|_{L^2}\leq C  (E^n_0E^n_2)^{\frac 14}(E^{n+1}_2)^{\frac 12}(D^{n+1}_2)^{\frac 12},\\
&|I''_2|\leq C\|\nabla^2 u^{n+1}\|_{L^4}\|\nabla\rho^n\|_{L^4}\|\nabla^2\rho^{n+1}\|_{L^2}\leq C  (E^n_1E^n_2)^{\frac 14}(E^{n+1}_2)^{\frac 34}(D^{n+1}_2)^{\frac 14},
\end{align*}
and
\begin{align*}
	|I''_3|\leq C(E^n_1)^{\frac 18}(D^n_2)^{\frac 38}(E^{n+1}_1)^{\frac 14}(E^{n+1}_2)^{\frac 34}.
\end{align*}
According to \eqref{ineq0} and \eqref{EE0} with $(E^{n+1}_1)^{\frac 14}\leq C\epsilon^{\frac14},$ one has
\begin{align*}
	|I_4^{''}|
	&\leq C\|{\rm div}u^{n+1}\|_{L^{\infty}}\|\nabla^2\rho^{n+1}\|_{	L^2}\|\nabla^2\rho^n\|_{L^2}\\ \notag
	&\leq C(E^{n+1}_1)^{\frac 14}(D^{n+1}_2)^{\frac 14}(E^{n+1}_2E^n_2)^{\frac12}\leq  C\epsilon^{\frac14}(E^n_2+E^{n+1}_2+D^{n+1}_2).
\end{align*}
Moreover, one can derive that
\begin{align*}
	|I''_5|+|I''_6|\leq C(E^n_1D^n_2)^{\frac 14}E^{n+1}_2,
\end{align*}
and
\begin{align*}
	&|I''_7|\leq C(E^n_0E^n_2)^{\frac 14}(D^{n+1}_2)^{\frac 12}\Big((D^{n+1}_1)^{\frac 12}+(D^{n+1}_2)^{\frac 12}+(E^{n+1}_2)^{\frac 12}\Big).
\end{align*}
The estimates of $\sum\limits_{i=8}^{10}I''_i,  (\sum\limits_{i=11}^{13}I''_i)$ can be handled  as  $I''_7, (I''_2)$,
\begin{align*}
	&|I''_8|\leq C(E^n_1E^n_2)^{\frac 14}(D^{n+1}_2)^{\frac 12}\Big((E^{n+1}_1E^{n+1}_2)^{\frac 14}+(D^{n+1}_1D^{n+1}_2)^{\frac 14}\Big),\\	&|I''_9|+|I''_{10}|\leq C(E^n_1)^{\frac 18}(D^n_2)^{\frac 38}(E^{n+1}_1E^{n+1}_2)^{\frac 14}(D^{n+1}_2)^{\frac 12},
\end{align*}
and
\begin{align*}
	&\sum\limits_{i=11}^{13}|I''_i|\leq C(E^n_1D^n_2)^{\frac 14}(E^{n+1}_2D^{n+1}_2)^{\frac 12}.
\end{align*}
Using Lemmas \ref{intef}, \ref{Lemma1}, we obtain
\begin{align*}
	|I''_{14}|\leq C\|g^{n+1}\|_{L^{\infty}(\mathcal{L}^2)}\|\nabla^3 u^n\|_{L^2}\|\nabla^2\nabla_Rg^{n+1}\|_{L^2(\mathcal{L}^2)}\leq C(E^{n+1}_0)^{\frac 12}\Big((D^{n+1}_2D^n_2)^{\frac 12}(E^{n+1}_2)^{\frac 12}\Big).
\end{align*}
Combining the estimates of $I''_1-I''_{14}$ and \eqref{initial1'},\eqref{initial1''}, we have
\begin{align}\label{E2}
	\frac 12\frac{d}{dt}E^{n+1}_2+\Big(1-C(\epsilon^{\frac 12}+\epsilon^{\frac 14}+\epsilon^{\frac 18})\Big)D^{n+1}_2
	&\leq C\Big(\epsilon^{\frac 14}+\epsilon^{\frac 18}+(\epsilon^{\frac 12}+\epsilon^{\frac 14}+\epsilon^{\frac 18})D^n_2\Big)E^{n+1}_2\\ \notag
	&~~~+C(\epsilon^{\frac14}+\epsilon^{\frac18}D^n_2)E^{n+1}_1+C\epsilon^{\frac 14}E^n_2+C\epsilon^{\frac 14}D^{n+1}_1.
\end{align}
 Denote $$B^{n+1}(t)=(E^{n+1}_1+E^{n+1}_2)(t)+\int_0^t(D^{n+1}_1+D^{n+1}_2)(t')dt'.$$ According to  \eqref{E1}-\eqref{E2} and using the conditions of \eqref{initial1'},\eqref{initial1''} yields 
\begin{align}\label{energy2}
	\Big(\frac 12-C[T(\epsilon^{\frac 14}+\epsilon^{\frac 18})+(\epsilon^{\frac 12}+\epsilon^{\frac 14}+\epsilon^{\frac 18})]\Big)B^{n+1}(t)\leq \frac{M}{2}+CT\epsilon^{\frac14},
\end{align}
with $T\leq\epsilon^{\frac 12}\le \frac{1}{32C}.$
Therefore, we obtain
\begin{align}\label{E2'}
	\sup\limits_{t\in [0,T]}B^{n+1}(t)\leq M.
\end{align}
Furthermore,  combining \eqref{EE0} and \eqref{E2'}, we arrive at \eqref{ee0},\eqref{ee1}.

\textbf{Step 3:}~$(\rho^n,u^n,g^n)\to (\rho,u,g)$ in $L^2\times L^2\times L^2(\mathcal{L}^2),n\to \infty$

Denote $$\bar{\rho}^{n+1}=\rho^{n+1}-\rho^{n}, ~~\bar{u}^{n+1}=u^{n+1}-u^{n}, ~~\bar{g}^{n+1}=g^{n+1}-g^{n}.$$
Then $(\bar{\rho}^{n+1},\bar{u}^{n+1},\bar{g}^{n+1})$ satisfy the following equations
\begin{align}\label{convergence}
	\left\{
	\begin{array}{ll}
		\bar{\rho}^{n+1}_t+{\rm div}\bar{u}^{n+1}=-\bar{\rho}^{n}{\rm div}u^{n}-u^{n}\cdot\nabla\bar{\rho}^{n+1}-\rho^{n}{\rm div}\bar{u}^{n+1}-\bar{u}^{n}\cdot\nabla\rho^{n}\triangleq \sum\limits_{i=1}^4l_i , \\
		\bar{u}^{n+1}_t+\nabla\bar{\rho}^{n+1}-{\rm div}\Sigma(\bar{u}^{n+1})-{\rm div}{\bar{\tau}}^{n+1}=-(h(\rho^{n})-h(\rho^{n-1})){\rm div}~\tau^{n}-(h(\rho^{n})\\ [1ex]
		~~~~~~-h(\rho^{n-1})){\rm div}\Sigma{(u^{n})}+(h(\rho^{n})-h(\rho^{n-1}))\nabla\rho^{n}  
	-u^{n}\cdot\nabla \bar{u}^{n+1}\\ [1ex]
	~~~~~~-\bar{u}^{n}\cdot\nabla u^{n}-h(\rho^{n}){\rm div}\Sigma{(\bar{u}^{n+1})}+h(\rho^{n}) \nabla\bar{\rho}^{n+1}-h(\rho^{n}) {\rm div}\bar{\tau}^{n+1}\triangleq \sum\limits_{i=5}^{12}l_i, \\[1ex]
		\bar{g}^{n+1}_t+\mathcal{L}\bar{g}^{n+1}+{\rm div}\bar{u}^{n+1}-\nabla \bar{u}^{n+1}R\nabla_R\mathcal{U}=-u^{n}\cdot\nabla \bar{g}^{n+1}\\ ~~~~~~-\frac 1 {\psi_\infty}\nabla_R\cdot(\sigma(u^{n})R\bar{g}^{n+1}\psi_\infty)-\bar{u}^{n}\cdot\nabla g^{n}-\frac 1 {\psi_\infty}\nabla_R\cdot(\sigma(\bar{u}^{n})Rg^{n}\psi_\infty)
		\triangleq\sum\limits_{i=13}^{16}l_i.  \\[1ex]
	\end{array}
	\right.
\end{align}
Multiplying  $\eqref{convergence}_1$,$\eqref{convergence}_2$,$\eqref{convergence}_3$ by $\bar{\rho}^{n+1},\bar{u}^{n+1},\bar{g}^{n+1}$ and using \eqref{e1}-\eqref{e3} yields
\begin{align*}
\frac 12 \frac {d} {dt}\bar{E}^{n+1}_0+\bar{D}^{n+1}_0=\sum\limits_{i=1}^4(l_i,\bar{\rho}^{n+1})+\sum\limits_{i=5}^{12}(l_i,\bar{u}^{n+1})+\sum\limits_{i=13}^{16}\left\langle l_i,\bar{g}^{n+1}\right\rangle\triangleq\sum\limits_{i=1}^{16}L_i.
\end{align*}
Using \eqref{ineq1'} and integrating by parts with $L_2=({\rm div}u^n,|\bar{\rho}^{n+1}|^2),$  we have 
\begin{align*}
	|L_1|+|L_2|&\leq C\|\nabla u^n\|_{L^{\infty}}\|\bar{\rho}^{n+1}\|_{L^2}(\|\bar{\rho}^{n}\|_{L^2}+\|\bar{\rho}^{n+1}\|_{L^2})\\ \notag
	&\leq C(E^n_1D^n_2)^{\frac14}(\bar{E}^{n+1}_0)^{\frac12}\Big((\bar{E}^{n}_0)^{\frac12}+(\bar{E}^{n+1}_0)^{\frac12}\Big).
\end{align*}
For $L_3,$ we get
\begin{align*}
	|L_3|\leq C(E^n_0E^n_2)^{\frac 14}(\bar{E}^{n+1}_0\bar{D}^{n+1}_0)^{\frac12}.
\end{align*}
Using  $\eqref{ineq0}$, one can obtain
\begin{align*}
	|L_4|\leq C \|\nabla\rho^n\|_{L^4}\|\bar{u}^n\|_{L^4}\|\bar{\rho}^{n+1}\|_{L^2}\leq C (E^n_1E^n_2)^{\frac 14}(\bar{E}^{n}_0\bar{D}^{n}_0)^{\frac 14}(\bar{E}^{n+1}_0)^{\frac 12}.
\end{align*}
The terms of $\sum\limits_{i=5}^{9}L_i$ can be estimated as follows. 
\begin{align*}
	&\sum\limits_{i=5}^7|L_i|\leq C(\bar{E}^n_0)^{\frac1 2}(\bar{E}^{n+1}_0)^{\frac 12}\Big((E^n_1E^n_2)^{\frac1 4}+(E^n_1)^{\frac1 8}(D^n_2)^{\frac38}+(E^n_1E^n_2D^n_1D^n_2)^{\frac 18}\Big),\\
	&|L_8|+|L_{9}|\leq C(E^n_1D^n_2)^{\frac 14}(\bar{E}^n_0\bar{E}^{n+1}_0)^{\frac 12}+C(E^n_0E^n_2)^{\frac 14}(\bar{E}^{n+1}_0\bar{D}^{n+1}_0)^{\frac 12}.
\end{align*}
Performing an integration  by parts and making use of Lemmas \ref{intef}, \ref{Lemma1}, we have
\begin{align}\label{u1}
	|L_{10}|
	&\leq |(\nabla h(\rho^n)\Sigma(\bar{u}^{n+1}),\bar{u}^{n+1})|+|(h(\rho^n)\Sigma(\bar{u}^{n+1}),\nabla\bar{u}^{n+1})|\\ \notag 
	&\leq C\|\nabla \bar{u}^{n+1}\|_{L^2}\Big(\|\nabla\rho^n\|_{L^4}\|\bar{u}^{n+1}\|_{L^4}+\|\rho^n\|_{L^{\infty}}\|\nabla\bar{u}^{n+1}\|_{L^2}\Big)\\ \notag 
	&\leq C(\bar{D}^{n+1}_0)^{\frac 12}\Big( (E^n_1E^n_2)^{\frac 14}(\bar{E}^{n+1}_0\bar{D}^{n+1}_0)^{\frac 14}+(E^n_0E^n_2)^{\frac 14}(\bar{D}^{n+1}_0)^{\frac 12}\Big),
\end{align}
and
\begin{align*}
	|L_{11}|&\leq C(E^n_1E^n_2)^{\frac 14}\bar{E}^{n+1}_0+ C(E^n_0E^n_2)^{\frac 14}(\bar{E}^{n+1}_0\bar{D}^{n+1}_0)^{\frac 12},~~\\ \notag
	|L_{12}|&\leq C(E^n_1E^n_2)^{\frac 14}(\bar{E}^{n+1}_0\bar{D}^{n+1}_0)^{\frac 12}+C(E^n_0E^n_2)^{\frac 14}\bar{D}^{n+1}_0.
\end{align*}
For $\sum\limits_{i=13}^{16}L_i,$ we have the following estimates:
\begin{align*}
	&|L_{13}|+|L_{14}|\leq C(E^n_1D^n_2)^{\frac 14}(\bar{E}^{n+1}_0\bar{D}^{n+1}_0)^{\frac12}\leq C\epsilon^{\frac 14}(D^n_2\bar{E}^{n+1}_0+\bar{E}^{n+1}_0+\bar{D}^{n+1}_0),~~\\
	&|L_{15}|\leq C(E^n_1E^n_2)^{\frac 14}(\bar{D}^{n+1}_0)^{\frac 12}(\bar{E}^n_0\bar{D}^n_0)^{\frac 14},~~~~|L_{16}|\leq C(E^n_0E^n_2)^{\frac 14}(\bar{D}^n_0\bar{D}^{n+1}_0)^{\frac 12}.
\end{align*}
Therefore, by adding up the estimates for $L_1-L_{16}$ and using again \eqref{initial1'}, \eqref{initial1''}, we conclude that
\begin{align}\label{convl2}
	\frac 12\frac{d}{dt}\bar{E}^{n+1}_0(t)+\bar{D}^{n+1}_0(t)&\leq C
\epsilon^{\frac 18}\bar{E}^{n+1}_0(t)+C\epsilon^{\frac18}\bar{D}^{n+1}_0(t)\\ \notag
&~~~+C\Big(\epsilon^{\frac18}+\epsilon^{\frac 18}(D^n_1+D^n_2)\Big)\bar{E}^n_0(t)+C\epsilon^{\frac 14}\bar{D}^n_0(t).
\end{align}
For convenience, we give the following notation $$\bar{A}^{n}(t)\triangleq\bar{E}^n_0(t)+\int_0^t\bar{D}^n_0(t')dt'.$$ It follows from \eqref{convl2} that
\begin{align*}
	\sup\limits_{t'\in[0,T]}\bar{A}^{n+1}(t)\leq  C(T+\epsilon^{\frac 18})\sup\limits_{t\in[0,T]}\bar{A}^{n}(t),
\end{align*}
where we used the fact that $\bar{E}^{n+1}_0(0)=0.$

Choosing $\epsilon,T$ sufficiently small and making use of the compactness argument, we conclude that the system \eqref{eq1} exists a  local solution with $\sup\limits_{t\in [0,T]}\mathcal{E}^{n+1}_0(t)+\int_0^T\mathcal{D}^{n+1}_0(t)dt\leq 2M.$

\textbf{Step 4:}~\textbf{Uniqueness}
\par
Assume  $(\rho_1,u_1,g_1)$ and $(\rho_2,u_2,g_2)$ are two solutions of \eqref{eq1}  with the same initial data. Then $\bar{\rho}=\rho_1- \rho_2,\bar{u}=u_1- u_2, \bar{g}=g_1- g_2$ satisfy the following equations
	\begin{align}\label{uniqueness}
	\left\{
	\begin{array}{ll}
	\bar{\rho}_t+{\rm div}\Sigma(\bar{u})=-u_1\cdot \nabla 	\bar{\rho}-{\rm div}u_1\bar{\rho} -{\rm div} \bar{u} \rho_2-\bar{u}\cdot \nabla\rho_2\triangleq\bar{F}, \\[1ex]
		\bar{u}_t+\nabla\bar{\rho}-{\rm div}\Sigma(\bar{u})-{\rm div}\bar{\tau}=-h(\rho_1){\rm div}\Sigma(\bar{u})+h(\rho_1)\nabla\bar{\rho}\\[1ex]~~~-h(\rho_1){\rm div}\bar{\tau}-(h(\rho_1)-h(\rho_2)){\rm div}\Sigma(\bar{u}_2)+(h(\rho_1)-h(\rho_2))\nabla\rho_2\\[1ex]~~~	 
		-(h(\rho_1)-h(\rho_2)){\rm div}\tau_2+u_1\cdot\nabla\bar{u}-\bar{u}\cdot\nabla u_2\triangleq\bar{H}, \\[1ex]
	\bar{g}_t+\mathcal{L}\bar{g}+{\rm div}\bar{u}-\nabla\bar{u}R\nabla_R\mathcal{U}=-u_1\cdot\nabla\bar{g}-\bar{u}\cdot\nabla g_2-\frac{1}{\psi_{\infty}}\nabla_{R}\cdot(\sigma(u_1)R\bar{g}\psi_{\infty})\\[1ex]
	~~~-\frac{1}{\psi_{\infty}}\nabla_{R}\cdot(\sigma(\bar{u})R{g_2}\psi_{\infty})\triangleq\bar{G}.  \\[1ex]
	\end{array}
	\right.
\end{align}
Taking $L^2$ inner product of \eqref{uniqueness}, we deduce that the following estimates:
\begin{align}\label{uq0}
		\frac 12 \frac{d}{dt}\|\bar{\rho}\|^2_{L^2}+({\rm div}\bar{u},\bar{\rho})\leq C\Big(\|\nabla u_1\|_{L^{\infty}}+\|\rho_2\|_{L^{\infty}}\Big)\|\bar{\rho}\|^2_{L^2}+C\|\bar{\rho}\|_{L^2}\|\nabla\rho_2\|_{L^4}\|\bar{u}\|_{L^4},
		\end{align}
	and
	\begin{align}\label{uq1}
			&\frac 1 2\frac{d}{dt}\|\bar{u}\|^2_{L^2}+(\nabla\bar{\rho},\bar{u})+\mu\|\nabla \bar{u}\|^2_{L^2}+(\mu+\mu')\|\nabla \bar{u}\|^2_{L^2}\\ \notag
			&\leq C\Big(\|\nabla\bar{u}\|_{L^2}+\|\bar{\rho}\|_{L^2}+\|\bar{\tau}\|_{L^2}\Big)\Big(\|\nabla\rho_1\|_{L^4}\|\bar{u}\|_{L^4}+\|\rho_1\|_{L^{\infty}}\|\nabla\bar{u}\|_{L^2}\Big)\\ \notag
			&~~~+C\|\bar{\rho}\|_{L^2}\|\bar{u}\|_{L^4}\Big(\|\nabla^2u_2\|_{L^4}+\|\nabla\rho_2\|_{L^4}+\|\nabla\tau_2\|_{L^4}\Big)+C\Big(\|\nabla u_1\|_{L^{\infty}}+\|\nabla u_2\|_{L^{\infty}}\Big)\|\bar{u}\|^2_{L^2}.
\end{align}
For the third equation of \eqref{uniqueness}, we have
\begin{align}\label{uq2}
		\frac {1}{2}\frac{d}{dt}\|\bar{g}\|^2_{L^2(\mathcal{L}^2)}+\|\nabla_R\bar{g}\|^2_{L^2(\mathcal{L}^2)}
		&\leq C\|\nabla u_1\|_{L^{\infty}}\|\bar{g}\|_{L^2(\mathcal{L}^2)}\|\nabla_R\bar{g}\|_{L^2(\mathcal{L}^2)}\\ \notag
		&~~~+C\|\nabla_R\bar{g}\|_{L^2(\mathcal{L}^2)}\Big(\|\bar{u}\|_{L^4}\|\nabla g_2\|_{L^4}+\|g_2\|_{L^{\infty}(\mathcal{L}^2)}\|\nabla\bar{u}\|_{L^2}\Big).
\end{align}
Denote
\begin{align*}
	&\bar{E}_0(t)=\|\bar{u}\|^2_{L^2}+\|\bar{\rho}\|^2_{L^2}+\|\bar{g}\|^2_{L^2(\mathcal{L}^2)},\\
	&\bar{D}_0(t)=\mu\|\nabla \bar{u}\|^2_{L^2}+(\mu+\mu')\|\nabla \bar{u}\|^2_{L^2}+\|\nabla_R\bar{g}\|^2_{L^2(\mathcal{L}^2)}.
\end{align*} 
Combining \eqref{uq0}-\eqref{uq2}, one has
\begin{align}\label{uq00}
\bar{E}_0(t)+\int_0^TD_0(t)dt\leq C\bar{E}_0(0)e^{\int_0^TC_1(t)dt},
\end{align}
where $C_1(t)=C\Big(\sum\limits_{i=1}^2(\|\nabla u_i\|_{L^{\infty}}+\|\rho_i\|_{L^{\infty}}+\|\nabla\rho_i\|_{L^4})+\|\nabla \tau_2\|_{L^4}\Big).$

According to \textbf{Step1-Step 3}, we infer that \eqref{eq1} exists a solution $(\rho,u,g)\in H^2\times H^2\times H^2(\mathcal{L}^2).$ Then, using  interpolation method and the fact that $\bar{E}_0(0)=0,$ we arrive at $(\rho_1,u_1,g_1)=(\rho_2,u_2,g_2).$

Therefore, according to \textbf{Step~1}-\textbf{Step~4}, we infer that the system \eqref{eq1} has a unique local solution. The proof of Proposition \ref{local existence } is completed.
	\end{proof}
\subsection{\textbf{Global regularity}}
	\par
	In this subsection, we give the proof of global solutions for \eqref{eq1} with large data by the dissipative term structure and the  interpolation method. Note that the dissipative of $(\rho,u,g)$ plays a key role and cannot be removed. Before giving the proof of the global solutions, we firstly give the global prior estimate of each derivative
	 (see Lemmas \ref{lemma1}, \ref{lemma2}, \ref{lemma3}). Then, using the bootstrap argument, we can obtain that the solutions exists global in time.

 Let the initial value satisfy \eqref{initial} and assume that  $(\rho,u,g)$ is a local  solution to the system \eqref{eq1}. In addition, assume further that $(\rho,u,g)$ satisfy the following conditions
\begin{align}\label{initial1}
E_1(t)+E_2(t)\leq 2M,~~and~~E_0(t)+E_1(t)\leq 2\epsilon(M).
\end{align}
For any $\beta\in \{0,1,2\},$ from $\int_B g\psi_{\infty}dR=0,$ one can obtain the following facts
\begin{align}\label{fa}
	&(\nabla^{\beta}{\rm div}u,\nabla^{\beta}\rho)+(\nabla^{\beta+1} \rho,\nabla^{\beta}u)=-(\nabla^{\beta}u,\nabla^{\beta+1}\rho)+(\nabla^{\beta+1} \rho,\nabla^{\beta}u)=0,\\\notag
	&\left\langle\nabla^{\beta}{\rm div}u,\nabla^{\beta}g\right\rangle=(-1)^{\beta}\int_{\mathbb{R}^2}\nabla^{2\beta}{\rm div}u\int_{B}g\psi_{\infty}dRdx=0, \\ \notag
&\left\langle\nabla^{\beta+1} uR\nabla_R\mathcal{U},\nabla^{\beta}g\right\rangle+(\nabla^{\beta}u,\nabla^{\beta}{\rm div}\tau)=(\nabla^{\beta+1} u,\nabla^{\beta}\tau)-(\nabla^{\beta+1}u,\nabla^{\beta}\tau)=0. 
\end{align}
Now  we present the following lemma.
\begin{lemm}\label{lemma1}
	Let $(\rho, u,g)$ be a local solution for \eqref{eq1} with the initial data $(\rho_0,u_0,g_0)$ under the conditions of Theorem \ref{global solution'} and satisfy \eqref{initial1}. Then, we have 
	\begin{align*}
		\frac{d}{dt}E_0(t)+D_0(t)\leq C(\epsilon^{\frac 12}+\epsilon^{\frac 14})(D_1(t)+\|\nabla\rho\|^2_{L^2}),	\end{align*}
	and
	\begin{align*}
	\frac{d}{dt} (u,\nabla \rho)+\frac 12\|\nabla \rho\|^2_{L^2}\leq C(\epsilon^{\frac 12}+\epsilon^{\frac 14})(D_0+D_1)(t)+{C}\Big(D_1(t)+\|\nabla u\|^2_{H^1}\Big).
	\end{align*}
\end{lemm}
\begin{proof}
Taking the $L^2$ inner product of \eqref{eq1} with $(\rho,u,g)$, using  $\eqref{fa}$, we infer that
	\begin{align}\label{ineq4}
		\frac{1}{2}	\frac{d}{dt}E_0(t)+D_0(t)=\sum\limits_{i=1}^8J_i,
	\end{align}
with
	\begin{align}\label{ineq5}
		\sum\limits_{i=1}^8J_i
		&\triangleq (\rho u,\nabla \rho)+(h(\rho)\nabla\rho,u)-(h(\rho){\rm div}\tau,u)-(u\cdot \nabla u,u)-\left\langle u\cdot\nabla g,g\right\rangle\\ \notag
		&~~~+(\nabla(h(\rho))\Sigma(u),u)+(h(\rho)\Sigma(u),\nabla u)+\left\langle\sigma(u)Rg,\nabla_Rg\right\rangle.
	\end{align}
Notice that
\begin{align}\label{luoz1}
	\|f\|_{L^4}\leq C\|f\|^{\frac 12}_{L^2}\|\nabla f\|^{\frac 12}_{L^2},~~\|\nabla f\|_{L^4}\leq C\|\nabla f\|^{\frac 12}_{L^2}\|\nabla^2 f\|^{\frac 12}_{L^2}.
\end{align}
According to $\eqref{luoz1}$ and Lemmas \ref{Lemma1}, \ref{luoz}, we have
\begin{align*}
		\sum\limits_{i=1}^4J_i\leq C\|u\|_{L^4}\|\rho\|_{L^4}(\|\nabla\rho\|_{L^2}+\|\nabla\tau\|_{L^2})+C\|u\|^2_{L^4}\|\nabla u\|_{L^2}\leq CE_0^{\frac 12}D^{\frac 14}_0\Big(\|\nabla\rho\|^{\frac 32}_{L^2}+D^{\frac 34}_0\Big),
\end{align*}
and
\begin{align*}
	J_5\leq C\| u\|_{L^4}\|g\|_{L^4(\mathcal{L}^2)}\|\nabla \nabla_Rg\|_{L^2(\mathcal{L}^2)}\leq C E^{\frac 12}_0D^{\frac 12}_0D^{\frac 12}_1.
\end{align*}
For $\sum\limits_{i=6}^8J_i,$ by $\eqref{luoz1}$ and Lemma \ref{luoz}, one has
\begin{align*}
&J_6+J_7\leq C\|\nabla u\|_{L^4}\Big(\|\nabla\rho\|_{L^2}\| u\|_{L^4}+\|\rho\|_{L^2}\| \nabla u\|_{L^4}\Big)\\ \notag&~~~~~~~~~~\leq C\|\nabla\rho\|_{L^2}E^{\frac 14}_0E^{\frac 14}_1D^{\frac 14}_0D^{\frac 14}_1+C\|\rho\|_{L^2}D^{\frac 12}_0D^{\frac 12}_1,\\
&J_8\leq C\|\nabla u\|_{L^4}\|g\|_{L^4(\mathcal{L}^2)}\| \nabla_Rg\|_{L^2(\mathcal{L}^2)}\leq CE^{\frac 14}_0E^{\frac 14}_1D_0^{\frac 34}D_1^{\frac 14}.
\end{align*}
Using the estimates of $J_1-J_8$ and the condition of \eqref{initial1}, we arrive at
\begin{align*}
	\frac12\frac{d}{dt}E_0+\Big(1-C(\epsilon^{\frac 12}+\epsilon^{\frac 14})\Big)D_0\leq C(\epsilon^{\frac 12}+\epsilon^{\frac 14})(D_1+\|\nabla\rho\|^2_{L^2}).
\end{align*}

	Note that $\|\nabla \rho\|^{ 2}_{L^2}$ appears on the right of the energy estimate in $E_0$, and there is no $\|\nabla \rho\|^{ 2}_{L^2}$ in $D_0$ to absorb it. In order to handle this difficulty, we multiply  $\nabla\rho$ to the second equation of \eqref{eq1} to generate a dissipative term on the density $\rho.$ Then, we deduce that 
	\begin{align}\label{ineq8}
		\frac{d}{dt}( u,\nabla \rho)+\|\nabla \rho\|^2_{L^2}=\|{\rm div}u\|^2_{L^2}+(u,\nabla F)+(H,\nabla \rho)+({\rm div}\tau,\nabla \rho)+({\rm div}\Sigma(u),\nabla\rho).
	\end{align}
	Integrating by parts and using \eqref{luoz1}, we have
	\begin{align*}
		(u,\nabla F)&=({\rm div}u,u\cdot\nabla\rho)+(|{\rm div}u|^2,\rho)\\ \notag
		&\leq C\|\nabla u\|_{L^4}\|u\|_{L^4}\|\nabla\rho\|_{L^2}+C\|\nabla^2 u\|_{L^2}\|\nabla u\|_{L^2}\|\rho\|_{L^4}\\ \notag
		&\leq CE^{\frac 14}_0E_1^{\frac 14}D^{\frac 14}_0D_1^{\frac 14}\|\nabla\rho\|_{L^2}+CE^{\frac 14}_0E^{\frac 14}_1D_0^{\frac 12}D_1^{\frac 12},
	\end{align*}
and
\begin{align*}
		&(H,\nabla \rho)+({\rm div}\tau,\nabla \rho)\\
		&\leq C \|\nabla\rho\|_{L^2}\Big(E^{\frac 14}_0E^{\frac 14}_1D^{\frac 14}_0D^{\frac 14}_1+E^{\frac 14}_0E^{\frac 14}_2D^{\frac 12}_1+E^{\frac 14}_0E^{\frac 14}_2\|\nabla\rho\|_{L^2}\Big)+CD_1+\frac 14\|\nabla\rho\|^2_{L^2}.
\end{align*}
Combining the above two estimates and \eqref{initial1}, we derive that
\begin{align}\label{endiss}
	\frac{d}{dt}(u,\nabla \rho)+\Big(\frac 34-C(\epsilon^{\frac 12}+\epsilon^{\frac 14})\Big)\|\nabla \rho\|^2_{L^2}\leq C(\epsilon^{\frac 12}+\epsilon^{\frac 14})(D_0+D_1)+{C}D_1+C\|\nabla u\|^2_{H^1}.
\end{align}
Then, the proof of Lemma \ref{lemma1} is finished.
\end{proof}
  In order to make $E_0(t)+E_1(t)\leq 2\epsilon$ hold at global time, the estimate of $E_1$ cannot involve $E_2$ and $D_2$. For example, using the interpolation method and the condition of \eqref{initial1}, we have $$(\nabla u\cdot\nabla\rho,\nabla\rho)\leq C\|\nabla u\|_{L^4}\|\nabla \rho\|_{L^4}\|\nabla\rho\|_{L^2}\leq C\epsilon^{\frac 14}\|\nabla^2 u\|_{L^2}^{\frac 12}\|\nabla\rho\|^{\frac 32}_{L^2}\leq C\epsilon^{\frac 14}D_1^{\frac 14}\|\nabla\rho\|^{\frac 32}_{L^2}.$$
 To make $E_1(t)+E_2(t)\leq 2M$ hold at global time, then  the estimate of $E_1(t)$ cannot involve $E_0(t)$ and $D_0(t)$. Therefore, we give two different estimates of $E_1(t)$ in following lemma.
\begin{lemm}\label{lemma2}
	Let $(\rho, u,g)$ be a local  solution for \eqref{eq1} and satisfy \eqref{initial1}. Assume that the initial data $(\rho_0,u_0,g_0)$ satisfy the conditions in Theorem \ref{global solution'}. Then, we have 
	\begin{align}\label{lemma2'}
	\frac{d}{dt}E_1(t)+D_1(t)\leq C\Big(\epsilon^{\frac 14}+\epsilon^{\frac 12}\Big)\|\nabla\rho\|^2_{L^2}.
	\end{align}
and
\begin{align}\label{lemma2''}
\frac{d}{dt}E_1(t)+D_1(t)\leq C(\epsilon^{\frac 14}+\epsilon^{\frac 12})\|\nabla^2\rho\|^2_{L^2}.
\end{align}
\end{lemm}
\begin{proof}
	Applying $\nabla$ to \eqref{eq1}, taking $L^2$ inner product and applying \eqref{fa}, we obtain
	\begin{align}\label{ineq11}
		\frac 1 2\frac{d}{dt}E_1(t)+D_1(t)=(\nabla F,\nabla\rho)+(\nabla H,\nabla u) +\left\langle\nabla G,\nabla g\right\rangle.
	\end{align}
	Integrating by parts, we have
	\begin{align*}
		(\nabla F,\nabla\rho)
		&=-(\nabla u\cdot \nabla\rho,\nabla\rho )+\frac 1 2({\rm div}u,|\nabla\rho|^2)-({\rm div}u\nabla\rho,\nabla\rho)-(\rho\nabla{\rm div}u,\nabla\rho)\triangleq\sum_{i=1}^{4}J'_i,\\
		(\nabla H,\nabla u)
		&=(u\cdot\nabla u,\nabla^2u)+(h(\rho) {\rm div}\Sigma{(u)},\nabla^2u)-(h(\rho)\nabla \rho,\nabla^2u)+(h(\rho) {\rm div}\tau,\nabla^2u)\triangleq\sum_{i=5}^{8}J'_i,\\
		\left\langle \nabla G,\nabla g\right\rangle
		&=	\left\langle\nabla\sigma(u)Rg,\nabla\nabla_Rg\right\rangle+\left\langle\sigma(u)R\nabla g,\nabla\nabla_Rg\right\rangle-\left\langle\nabla u\cdot\nabla g,	\nabla g\right\rangle+\frac 1 2\left\langle{\rm div}u,|\nabla g|^2\right\rangle\triangleq\sum_{i=9}^{12}J'_i.
	\end{align*}
Now we estimate the above terms. Using Lemma \ref{intef}, one has
\begin{align}\label{luoz2}
&	\|f\|_{L^{\infty}}\leq C\|f\|^{
	\frac 12}_{L^2}\|\nabla^2f\|^{
	\frac 12}_{L^2},~~\|\nabla f\|_{L^4}\leq C\|f\|^{\frac 14}_{L^2}\|\nabla^2f\|^{\frac 34}_{L^2},\\ \notag
&\|\nabla f\|_{L^2}\leq C\|f\|^{\frac 12}_{L^2}\|\nabla^2f\|^{\frac 12}_{L^2},~~\| f\|_{L^4}\leq C\|f\|^{\frac 34}_{L^2}\|\nabla^2f\|^{\frac 14}_{L^2}.
\end{align}
	Using $\eqref{luoz1}$, we infer that
	\begin{align*}
		\sum_{i=1}^3{J}'_i\leq C\|\nabla u\|_{L^4}\|\nabla\rho\|_{L^4}\|\nabla \rho\|_{L^2}\leq C(E_1E_2)^{\frac 14}D^{\frac 14}_1\|\nabla\rho\|^{\frac 32}_{L^2}.
	\end{align*}
where we use $\|\nabla u\|^{\frac 12}_{L^2}\|\nabla^2\rho\|^{\frac 12}_{L^2}\leq C(E_1E_2)^{\frac 14}$ and $\|\nabla^2u\|^2_{L^2}\leq CD_1.$ 

It follows from \eqref{luoz2} and Lemmas \ref{Lemma1}, \ref{luoz} that
\begin{align*}
		&J'_4\leq C\|\rho\|_{L^{\infty}}\|\nabla^2u\|_{L^2}\|\nabla\rho\|_{L^2}\leq C(E_0E_2)^{\frac 14}D_1^{\frac 12}\|\nabla\rho\|_{L^2},\\ \notag
		&	J'_5\leq C\|\nabla^2u\|_{L^2}\|u\|_{L^4}\|\nabla u\|_{L^4}\leq CE^{\frac 12}_0D_1,
	\end{align*}
and
\begin{align*}
		&\sum\limits_{i=6}^7J'_i\leq C(E_0E_2)^{\frac 14}D^{\frac 12}_1\Big(D^{\frac 12}_1+\|\nabla\rho\|_{L^2}\Big),~~~J'_8+ J'_{9}\leq  C(E_0E_2)^{\frac 14}D_1.
\end{align*}
Using \eqref{luoz1} and Lemma \ref{Lemma1}, we get
\begin{align*}
	\sum\limits_{i=10}^{12}J'_i\leq C\|\nabla u\|_{L^4} \|\nabla g\|_{L^4(\mathcal{L}^2)}\|\nabla\nabla_R g\|_{L^2(\mathcal{L}^2)}\leq C(E_1E_2)^{\frac 14}D_1,
\end{align*}
where we use $\|\nabla u\|^{\frac 12}_{L^2}\|\nabla^2g\|^{\frac 12}_{L^2(\mathcal{L}^2)}\leq C(E_1E_2)^{\frac 14}.$
 
 Combining estimates of $J'_1-J'_{12}$ and \eqref{initial1}, we conclude that 
 \begin{align}\label{en1}
 	\frac12\frac{d}{dt}E_1(t)+\Big(1-C(\epsilon^{\frac 14}+\epsilon^{\frac 12})\Big)D_1(t)\leq C\Big(\epsilon^{\frac 14}+\epsilon^{\frac 12}\Big)\|\nabla\rho\|^2_{L^2}.
 \end{align}
On the other hand, using $\eqref{luoz2}$, we estimate $J'_1-J'_3$, $J'_4$ and $J'_7$ as follows
\begin{align*}
	&\sum_{i=1}^3{J}'_i\leq C\|u\|^{\frac12}_{L^2}\|\rho\|^{\frac12}_{L^2}\|\nabla^2u\|^{\frac 12}_{L^2}\|\nabla^2\rho\|^{\frac 32}_{L^2}\leq C\epsilon^{\frac 12}D^{\frac 14}_1\|\nabla^2\rho\|^{\frac 32}_{L^2},\\ \notag
	&	J'_4+J'_7\leq C\|\rho\|_{L^2}\|\nabla^2u\|_{L^2}\|\nabla^2\rho\|_{L^2}\leq C\epsilon^{\frac 12}D^{\frac 12}_1\|\nabla^2\rho\|_{L^2}.
\end{align*}
Hence, we obtain
\begin{align}\label{en1'}
	\frac12\frac{d}{dt}E_1(t)+\Big(1-C(\epsilon^{\frac 14}+\epsilon^{\frac 12})\Big)D_1(t)\leq C(\epsilon^{\frac 14}+\epsilon^{\frac 12})\|\nabla^2\rho\|^2_{L^2},
\end{align}
which will be useful to prove a different closed energy estimate by combining the estimate for $E_2(t)+\eta(\nabla u,\nabla^2\rho)$.
Then, we finish the proof of Lemma \ref{lemma2}.
\end{proof}
Making use of the dissipative structure of $(\rho,u,g)$ and combining the energy of $E_0(t),E_1(t)$  and $(u,\nabla\rho),$ we get the first closed energy estimate as follows.
\begin{rema}\label{re1}
	Combining Lemmas \ref{lemma1}, \ref{lemma2}, for any $t\in [0,\infty),$ we have 
	\begin{align*}
		E_0(t)+E_1(t)+2\eta(u,\nabla\rho)+\int_0^t\Big(D_0+D_1+2\eta\|\nabla\rho\|^2_{L^2}\Big)(t')
	dt'\leq E_0(0)+E_1(0).
	\end{align*}
\end{rema}
To obtain the optimal decay rate in $H^2$, we give a more precise higher order derivatives estimate in the proof of global existence for strong solutions to \eqref{eq1}. The following lemma is the energy estimate of $E_2(t).$
\begin{lemm}\label{lemma3}
	Let $(\rho, u,g)$ be a local  solution for \eqref{eq1} with the initial data $(\rho_0,u_0,g_0)$ under the conditions in Theorem \ref{global solution'} and satisfy \eqref{initial1}. Then, we have
	\begin{align}\label{lemma3'}
	\frac{d}{dt}E_2(t)+D_2(t)\leq C\Big((E_0E_2)^{\frac 14}+(E_1E_2)^{\frac 14}+E_1^{\frac 38}E_2^{\frac 18}+E_1^{\frac 12}\Big)\Big(\|\nabla^2\rho\|^2_{L^2}+D_1 \Big),
	\end{align}
	and
	\begin{align}\label{lemma3''}	\frac{d}{dt}(\nabla u,\nabla^2 \rho)+\frac12\|\nabla^2 \rho\|^2_{L^2}\leq C(\epsilon^{\frac 14}+\epsilon^{\frac 12})(D_1+D_2)+{C}\Big(D_2+\|\nabla^2 u\|^2_{H^1}\Big).
	\end{align}
\end{lemm}
\begin{proof}
	Applying $\nabla^2$ to \eqref{eq1}, taking $L^2$ inner product and using \eqref{fa}, we obtain the following estimate
	\begin{align}\label{ineq14}
		&~~~\frac 1 2\frac{d}{dt}E_2(t)+D_2(t)= (\nabla^2 F,\nabla^2\rho)+(\nabla^2 H,\nabla^2 u)+\left\langle \nabla^2 G,\nabla^2 g\right\rangle,
	\end{align}
	where
	\begin{align}\label{ineq16'}
		(\nabla^2 F,\nabla^2 \rho)
		=&-2(\nabla u\cdot\nabla^2\rho,\nabla^2 \rho)+\frac{1}{2}({\rm div}u,|\nabla^2 \rho|^2)-(\rho\nabla^2{\rm div}u,\nabla^2 \rho)\\ \notag &-2(\nabla \rho \nabla{\rm  div}u,\nabla^2 \rho)-(\nabla^2u\cdot\nabla\rho,\nabla^2 \rho)\triangleq\sum_{i=1}^{5}J''_i.
	\end{align}
	Integrating by parts, we get $(\nabla^2 H,\nabla^2 u)=-(\nabla H,\nabla^3 u).$
	
	 Hence, we have
	\begin{align}\label{ineq16}
		(\nabla^2 H,\nabla^2 u)=&(h(\rho)\nabla {\rm div}\Sigma (u),\nabla^3u)-(h(\rho)\nabla^2\rho,\nabla^3u)+(h(\rho)\nabla{\rm div}\tau,\nabla^3u)\\ \notag 
		&+(u\cdot\nabla^2u,\nabla^3u)+(\nabla u\cdot\nabla u,\nabla^3u)-(\nabla (h(\rho))\nabla \rho,\nabla^3u)\\ \notag 
		&+(\nabla (h(\rho)) {\rm div}\Sigma (u),\nabla^3u)+(\nabla (h(\rho)){\rm div}\tau,\nabla^3u)\triangleq\sum_{i=6}^{13}J''_i,
	\end{align}
	and
	\begin{align}\label{ineq16''}
		\left\langle\nabla^2G,\nabla^2g\right\rangle=&-\left\langle\nabla^2u\cdot\nabla g,\nabla^2g\right\rangle+2\left\langle\nabla\sigma(u)R\nabla g,\nabla^2\nabla_Rg\right\rangle+\left\langle\nabla^2\sigma(u)Rg,\nabla^2\nabla_Rg\right\rangle\\ \notag 
		&-2\left\langle\nabla u\cdot\nabla^2 g,\nabla^2g\right\rangle+\frac 1 2\left\langle{\rm div}u,|\nabla^2g|^2\right\rangle+\left\langle\sigma(u)R\nabla^2g,\nabla^2\nabla_Rg\right\rangle\triangleq\sum_{i=14}^{19}J''_i.
	\end{align}
	Now, we estimate the above terms. It follows from Lemma \ref{intef} that 
	\begin{align}\label{b0}
		\|\nabla f\|_{L^{\infty}}\leq C	\|\nabla f\|^{\frac 1 2}_{L^{2}}\|\nabla^3 f\|^{\frac 1 2}_{L^{2}},~~
		\|\nabla^2f\|_{L^4}\leq C\|\nabla f\|^{\frac  1 4}_{L^2}\|\nabla^3 f\|^{\frac  3 4}_{L^2}.
	\end{align}
According to \eqref{luoz2},\eqref{b0}, we obtain
\begin{align*}
	J''_1+J''_2+J''_3
	&\leq C\|\nabla u\|_{L^{\infty}}\|\nabla^2\rho\|^2_{L^2}+C\|\rho\|_{L^{\infty}}\|\nabla^3u\|_{L^2}\|\nabla^2\rho\|_{L^2}\\ \notag
	&\leq C(E_1E_2)^{\frac 14}D^{\frac 14}_2\|\nabla^2\rho\|^{\frac 32}_{L^2}+C(E_0E_2)^{\frac 14}D^{\frac 12}_2\|\nabla^2\rho\|_{L^2},
\end{align*}
with $\|\nabla u\|^{\frac12}_{L^{2}}\|\nabla^2 \rho\|^{\frac12}_{L^{2}}\leq C(E_1E_2)^{\frac 14}$ and $\|\nabla^3u\|^{\frac 12}_{L^2}\leq CD^{\frac 14}_2.$

Using  $\eqref{luoz2}, \eqref{b0}$ again, we have
\begin{align*}
	J''_4+J''_5\leq C\|\nabla\rho\|_{L^4}\|\nabla^2u\|_{L^4}\|\nabla^2\rho\|_{L^2}\leq CE^{\frac 38}_1E^{\frac 18}_2D^{\frac 38}_2\|\nabla^2\rho\|^{\frac 54}_{L^2},
\end{align*}
and
\begin{align*}
	\sum\limits_{i=6}^9J''_i&\leq C\|\nabla^3u\|_{L^2}\Big[\|\rho\|_{L^{\infty}}\Big(\|\nabla^3u\|_{L^2}+\|\nabla^2\rho\|_{L^2}+\|\nabla^2\nabla_Rg\|_{L^2(\mathcal{L}^2)}\Big)+\|u\|_{L^{\infty}}\|\nabla^2u\|_{L^2}\Big]\\ \notag
	&\leq C(E_0E_2)^{\frac 14}D^{\frac12}_2\Big(D^{\frac12}_2+\|\nabla^2\rho\|_{L^2}\Big)+C(E_0E_2)^{\frac 14}(D_1D_2)^{\frac 12},
\end{align*}
where $\|\nabla u\|^{\frac12}_{L^2}\|\nabla \rho\|^{\frac14}_{L^2}\|\nabla^2\rho\|^{\frac 14}_{L^2}\leq CE^{\frac 38}_1E^{\frac 18}_2.$

By Lemmas \ref{intef}, \ref{luoz}, one has
\begin{align}\label{tau}
	\|\nabla\tau\|_{L^4}\leq C\|\nabla g\|^{\frac14}_{L^2(\mathcal{L}^2)}\|\nabla^2 g\|^{\frac14}_{L^2(\mathcal{L}^2)}\|\nabla\nabla_R g\|^{\frac14}_{L^2(\mathcal{L}^2)}\|\nabla^2\nabla_R g\|^{\frac14}_{L^2(\mathcal{L}^2)}.
\end{align}
This together with $\eqref{luoz2},\eqref{b0}$ and \eqref{tau} yields that
\begin{align*}
	\sum\limits_{i=10}^{13}J''_i
	&\leq C\|\nabla^3u\|_{L^2}\Big(\|\nabla u\|^2_{L^4}+\|\nabla \rho\|^2_{L^4}+\|\nabla\rho\|_{L^4}\|\nabla^2u\|_{L^4}+\|\nabla\rho\|_{L^4}\|\nabla\tau\|_{L^4}\Big)\\ \notag
	&\leq CD^{\frac 12}_2\Big[E^{\frac 12}_1(D^{\frac 12}_1+\|\nabla^2\rho\|_{L^2})+\|\nabla^2\rho\|^{\frac 12}_{L^2}\Big((E_1E_2D_2)^{\frac 14}+E_1^{\frac 38}E^{\frac18}_2D^{\frac18}_1D_2^{\frac 18}\Big)\Big],
\end{align*}
and 
\begin{align*}
	J''_{14}+	J''_{15}
	&\leq \|\nabla g\|_{L^4(\mathcal{L}^2)}\|\nabla^2u\|_{L^4}\|\nabla^2\nabla_Rg\|_{L^2(\mathcal{L}^2)}\leq C(E_1E_2)^{\frac 14}D_2.
\end{align*}
For $\sum\limits_{i=16}^{19}J''_i,$ we have the following estimate
\begin{align*}
\sum\limits_{i=16}^{19}J''_i
	&\leq C\|\nabla^2\nabla_Rg\|_{L^2(\mathcal{L}^2)}\Big(\|g\|_{L^{\infty}(\mathcal{L}^2)}\|\nabla^3u\|_{L^2}+\|\nabla u\|_{L^{\infty}}\|\nabla^2g\|_{L^2(\mathcal{L}^2)}\Big)\\ \notag 
	&\leq CD^{\frac 12}_2\Big((E_0E_2)^{\frac 14}D^{\frac 12}_2+(E_1E_2)^{\frac 14}D^{\frac 12}_2\Big).
\end{align*}
	Inserting the estimates of $\sum\limits_{i=1}^{19}J''_i$ into \eqref{ineq14}, we get
	\begin{align}\label{E2''}
		&\frac{d}{dt}E_2(t)+\Big(1-C(\epsilon^{\frac 14}+\epsilon^{\frac 38}+\epsilon^{\frac 12})\Big)D_2(t)\\ \notag
		&\leq C\Big((E_0E_2)^{\frac 14}+(E_1E_2)^{\frac 14}+E_1^{\frac 38}E_2^{\frac 18}+E_1^{\frac 12}\Big)\Big(\|\nabla^2\rho\|^2_{L^2}+D_1 \Big),
	\end{align}
	which implies that \eqref{lemma2'} holds.
	
	On the other hand, we have 
	\begin{align}\label{B7}
		\frac{d}{dt}(\nabla u,\nabla^2 \rho)+\|\nabla^2 \rho\|^2_{L^2}
		&= \|\nabla{\rm div}u\|^2_{L^2}+(\nabla u,\nabla^2F)+(\nabla H,\nabla^2\rho)\\ \notag
		&~~~+(\nabla{\rm div}\tau,\nabla^2\rho)+(\nabla{\rm div}\Sigma(u),\nabla^2\rho).
	\end{align}
Using Lemmas \ref{intef}-\ref{luoz} and integrating by parts, we obtain
	\begin{align}\label{b7}
	(\nabla^3 u,F)
	&\leq C\|\nabla^3 u\|_{L^2}\Big(\|u\|_{L^4}\|\nabla\rho\|_{L^4}+\|\rho\|_{L^4}\|\nabla u\|_{L^4}\Big)\\ \notag
	&
		\leq C\|\nabla^3 u\|_{L^2}\Big(\|u\|^{\frac 34}_{L^2}\|\rho\|^{\frac 14}_{L^2}\|\nabla^2u\|^{\frac 14}_{L^2}\|\nabla^2\rho\|^{\frac 34}_{L^2}+\|\rho\|^{\frac 34}_{L^2}\|u\|^{\frac 14}_{L^2}\|\nabla^2\rho\|^{\frac 14}_{L^2}\|\nabla^2u\|^{\frac 34}_{L^2}\Big)\\ \notag
		&
		\leq CE^{\frac 12}_0D^{\frac 12}_2\Big(D^{\frac 18}_1\|\nabla^2\rho\|^{\frac 34}_{L^2}+D^{\frac 38}_1\|\nabla^2\rho\|^{\frac 14}_{L^2}\Big),
	\end{align}
	and
	\begin{align}\label{b7'}
			(\nabla H,\nabla^2\rho)&\leq
			 C\|\nabla^2\rho\|_{L^2}\Big[\|\rho\|_{L^{\infty}}\Big(\|\nabla^3 u\|_{L^2}+\|\nabla^2 \tau\|_{L^2}+\|\nabla^2\rho\|_{L^2}\Big)+\|u\|_{L^{\infty}}\|\nabla^2 u\|_{L^2}\Big]\\ \notag
			&~~~+C\|\nabla^2\rho\|_{L^2}\Big[\|\nabla\rho\|_{L^4}\Big(\|\nabla^2u\|_{L^4}+\|\nabla\tau\|_{L^4}+\|\nabla\rho\|_{L^4}\Big)+\|\nabla u\|^2_{L^4}\Big]\\ \notag
			&\leq C(E_0E_2)^{\frac 14}\|\nabla^2\rho\|_{L^2}\Big(D^{\frac 12}_2+D^{\frac 12}_1+\|\nabla^2\rho\|_{L^2}\Big)\\ \notag
			&~~~+C\|\nabla^2\rho\|_{L^2}\Big((E_1E_2)^{\frac 14}D^{\frac 14}_2\|\nabla^2\rho\|^{\frac 12}_{L^2}+E_1^{\frac 12}D^{\frac 14}_2\|\nabla^2\rho\|^{\frac 12}_{L^2}+E_1^{\frac 12}\|\nabla^2\rho\|_{L^2}+E^{\frac 12}_1D^{\frac 12}_1\Big).
	\end{align}
H{\"o}lder's inequality and Lemma \ref{Lemma1} yields 
\begin{align}\label{b7''}
	&(\nabla{\rm div}\tau,\nabla^2\rho)+(\nabla{\rm div}\Sigma(u),\nabla^2\rho)\\ \notag
	&\leq \|\nabla^2\rho\|_{L^2}(\|\nabla^2\tau\|_{L^2}+\|\nabla^3u\|_{L^2})\leq {C}\Big(D_2+\|\nabla^3 u\|^2_{L^2}\Big)+\frac14 \|\nabla^2\rho\|^2_{L^2} .
\end{align}
This together with \eqref{b7}-\eqref{b7''} and \eqref{initial1},  we arrive at 
	\begin{align}
		&\frac{d}{dt}(\nabla u,\nabla^2 \rho)+\Big(\frac 34-C(\epsilon^{\frac14}+\epsilon^{\frac12})\Big)\|\nabla^2 \rho\|^2_{L^2}\leq C(\epsilon^{\frac 14}+\epsilon^{\frac 12})(D_1+D_2)+{C}\Big(D_2+\|\nabla^2u\|^2_{H^1}\Big).
	\end{align}
Then, we finish the proof of Lemma \ref{lemma3}.
\end{proof}
According to the energy $E_1(t),E_2(t)$ and the dissipative structures of $(\rho,u,g)$, we obtain  another closed energy estimate, which will be used to prove optimal decay rate of $\dot{H}^1$ of $(\rho,u).$
\begin{rema}\label{re2}
	Combining Lemmas \ref{lemma2}, \ref{lemma3}. Then, for any $t\in[0,\infty),$ we have
	\begin{align*}
		\mathcal{E}_1(t)+\int_0^t\mathcal{D}_1(t')dt'\leq C\mathcal{E}_1(0).
	\end{align*}
\end{rema}
We now come to the proof of global existence of solutions for \eqref{eq1}.\\
\textbf{Proof of Theorem \ref{global solution'}:}
\begin{proof}
Combining Remarks \ref{re1}, \ref{re2}, we deduce that $\eqref{initial1}$ is valid for any $t\in[0,\infty)$ and  
\begin{align*}
		\sup_{t\in [0,\infty)} \mathcal{E}_0(t)+\int_0^{\infty} \mathcal{D}_0(t)dt\leq C\mathcal{E}_0(0).
\end{align*}
	Therefore, the proof of Theorem \ref{global solution'} is finished.
\end{proof}

The rest of this paper is devoted to proving the decay rate of $\|\nabla^{s+1}(\rho,u)\|_{L^2}, \|\nabla^s g\|_{L^2(\mathcal{L}^2)}$ with $s\in [0,1],$ where optimal decay rate of $\|\nabla^2(\rho,u)\|^2_{L^2}$ is a new result.

\subsection{\textbf{Optimal decay estimates}}
\par

In this subsection, we are devoted to studying the decay rate of the derivatives of each order of $(\rho,u,g)$ in $L^2,$ which mainly depends on Schonbek's strategy.  Firstly, applying the Fourier splitting method and taking Fourier transform, one can get the initial  logarithmic decay rate with $\mathcal{E}_0
\leq C{\rm ln}^{-l}(e+t),l\in\mathbb{N}.$ Then, making use of the time weighted energy estimate and  $\mathcal{E}_0
\leq C{\rm ln}^{-l}(e+t)$, we improve the decay rate to $\mathcal{E}_0
\leq C(1+t)^{-\frac 12}.$ By using  the Fourier splitting method and the Littlewood-Paley decomposition theory, we obtain the decay rate of $\|\nabla^s(\rho,u)\|_{L^2}+\|\nabla^sg\|_{L^2(\mathcal{L}^2)}$ with $s\in [0,1].$ Since the decay rate of $\nabla^sg$ is the same as that of $\nabla^s\nabla u,$  we  improve the decay rate of  $\|g\|_{L^2(\mathcal{L}^2)}$ to $(1+t)^{-1}$. Owing to $ \|\nabla^2(\rho,u)\|^2_{L^2}+\eta (\nabla u,\nabla^2\rho)\neq \|\nabla^2(\rho,u)\|^2_{L^2},$  we fail to obtain optimal decay rate of $\|\nabla^2(\rho,u)\|_{L^2}$ by applying the same method. To obtain the optimal decay rate of $\|\nabla^2(\rho,u)\|_{L^2}\leq C(1+t)^{-\frac 32},$ we need to use the improved Fourier
splitting method and the time weighted energy estimate.  Finally, we obtain the optimal decay rate of $\|\nabla g\|_{L^2(\mathcal{L}^2)}$ of $(1+t)^{-\frac 32}$.

For the special dimension  two, we fail to get the initial polynomial decay rate by using the bootstrap argument. However, we prove logarithmic decay rate by the method in \cite{Luo-Yin}. The similar proof of the following Lemma \ref{decayln} can be found in \cite{LZN.com.2021Global}, we thus omit the details here.
\begin{lemm}\cite{LZN.com.2021Global}\label{decayln}
	Under the assumptions of Theorem \ref{opdecay'}, for any $l\in \mathbb{N},$ it holds that
	\begin{align}\label{decaylnl}
		\mathcal{E}_0(t)\leq C\ln^{-l}(e+t),~~and ~~\mathcal{E}_1(t)\leq C(1+t)^{-1}\ln^{-l}(e+t).
	\end{align}
\end{lemm}

\begin{lemm}\label{decayp}
Under the conditions in Theorem \ref{opdecay'}, then there has a constant $C$ such that
	\begin{equation}\label{p1}
		\mathcal{E}_0(t)	\leq C(1+t)^{-\frac{1}{2}},~~and~~
		\mathcal{E}_1(t)\leq  C(1+t)^{-\frac{3}{2}}.
	\end{equation}
	Moreover, it holds
	\begin{align}\label{inte}
		\int_0^t(1+t')^{\alpha_1}\mathcal{D}_1(t')dt'\leq C,
	\end{align}
	with $\alpha_1\in [0,\frac32).$
\end{lemm}
\begin{proof}
	We define $S(t)=\{|\xi|: |\xi|^2<C_2(1+t)^{-1}\}$ with $C_2$ large enough. Making use of the Schonbek's strategy, we split the phase space into time-dependent domain  as follows:
	\begin{align*}
		\|\nabla u\|^2_{H^2}=\int_{S(t)}(1+|\xi|^4)|\xi|^2|\hat{u}(\xi)|^2d\xi+\int_{S^c(t)}(1+|\xi|^4)|\xi|^2|\hat{u}(\xi)|^2d\xi.
	\end{align*}
	It follows from the definition of  $S(t)$ and Theorem \ref{opdecay'} that
	\begin{align}\label{e00}
		\frac{d}{dt}\mathcal{E}_0(t)+\frac{C_2\mu}{1+t}\|u\|^2_{H^2}+\frac{\eta C_2}{2(1+t)}\|\rho\|^2_{H^2}+\frac{1}{2}\|\nabla_Rg\|^2_{H^2(\mathcal{L}^2)}\leq \frac{CC_2}{1+t} \int_{S(t)}|\hat{u}(\xi)|^2+|\hat{\rho}(\xi)|^2d\xi.
	\end{align}
	Next, we focus on the  estimate of $(\rho,u)$ over $S(t).$
	Applying Fourier transform to \eqref{eq1}, we infer that
	\begin{align}\label{fl}
		\left\{
		\begin{array}{ll}
			\hat{\rho}_t+i\xi_{k} \hat{u}^k=\hat{F},  \\[1ex]
			\hat{u}^{j}_t+\mu|\xi|^2 \hat{u}^j+(\mu+\mu')\xi_{j} \xi_{k} \hat{u}^k+i\xi_{j} \hat{\rho}-i\xi_k\hat{\tau}^{jk}=\hat{H}^j,  \\[1ex]
			\hat{g}_t+\mathcal{L}\hat{g}-i\xi_k\hat{u}^jR_j\partial_{R_k}\mathcal{U}+i\xi_k\hat{u}^k=\hat{G}. \\[1ex]
		\end{array}
		\right.
	\end{align}
 According to the above inequalit                                 ies and \eqref{fl}, we obtain the following estimate
	\begin{align}\label{ineq26}
		&\frac 1 2 \frac{d}{dt}(|\hat{\rho}|^2+|\hat{u}|^2+\|\hat{g}\|^2_{\mathcal{L}^2})+\mu|\xi|^2|\hat{u}|^2+(\mu+\mu')|\xi\cdot\hat{u}|^2+\frac{1}{2}\|\nabla_R\hat{g}\|^2_{\mathcal{L}^2}\\ \notag 
		&\leq C\Big(\mathcal{R}e[\hat{F}\bar{\hat{\rho}}]+\mathcal{R}e[\hat{H}\bar{\hat{u}}]+\mathcal{R}e[\hat{G}\bar{\hat{g}}]\Big) ,
	\end{align}
	where we use $\mathcal{R}[i\xi\cdot\hat{u}\bar{\hat{\rho}}]+\mathcal{R}[i\xi\hat{\rho}\cdot\bar{\hat{u}}]=\mathcal{R}_e[i\xi\otimes\bar{\hat{u}}:\hat{\tau}]+\mathcal{R}_e[i\xi\otimes\hat{u}:\bar{\hat{\tau}}]=0.$
	
	Therefore, we deduce that
	\begin{align}\label{f0}
		\int_{S(t)}|\hat{\rho}|^2+|\hat{u}|^2+\|\hat{g}\|^2_{\mathcal{L}^2}d\xi
		&\leq C\int_{S(t)}|\hat{\rho}_0|^2+|\hat{u}_0|^2+\|\hat{g}_0\|^2_{\mathcal{L}^2}d\xi+C\int_0^t\int_{S(t)}|\hat{F}\bar{\hat{\rho}}|+|\hat{H}\bar{\hat{u}}|d\xi ds\\ \notag &~~~+C\int_0^t\int_{S(t)}\int_{B}\psi_\infty|\mathcal{F}(u\cdot\nabla g)|^2+\psi_\infty|\mathcal{F}(\nabla uRg)|^2 dRd\xi dt'.
	\end{align}
It follows from Theorem \ref{opdecay'} that
	\begin{align}\label{f00}
		\int_{S(t)}|\hat{\rho}_0|^2d\xi
		&\leq C\sum_{j\leq \log_2[\frac {4} {3}C_2^{\frac 1 2 }{\frac{1}{\sqrt{1+t}}}]}\int_{\mathbb{R}^{2}} 2\varphi^2(2^{-j}\xi)|\hat{\rho}_0|^2d\xi\\ \notag 
		&\leq C\sum_{j\leq \log_2[\frac {4} {3}C_2^{\frac 1 2 }\frac{1}{\sqrt{1+t}}]}\int_{\mathbb{R}^{2}}\|\dot{\Delta}_j \rho_0\|^2_{L^2}d\xi\\ \notag
		&\leq C(1+t)^{-1}\|\rho_0\|^2_{\dot{B}^{-1}_{2,\infty}},
	\end{align}
	and
	\begin{align}\label{f001}
	\int_{S(t)}|\hat{u}_0|^2+|\hat{g}_0|^2_{\mathcal{L}^2}d\xi&\leq C(1+t)^{-1}(\|u_0\|^2_{\dot{B}^{-1}_{2,\infty}}+\|g_0\|^2_{\dot{B}^{-1}_{2,\infty}(\mathcal{L}^2)}).
	\end{align}
Thanks to $|\mathcal{F}({\rm div}(\rho u))\bar{\hat{\rho}}|\leq C|\xi
||\mathcal{F}((\rho u))\bar{\hat{\rho}}|,$  we have
\begin{align}\label{flf}
\int_{S(t)}\int_0^t|\mathcal{F}({\rm div}(\rho u))\bar{\hat{\rho}}|dt'd\xi
&\leq C\Big(\int_{S(t)}|\xi|^2d\xi\Big)^{\frac 12}\Big(\int_0^t\|\rho\|_{L^2}\|\rho u\|_{L^1}dt'\Big)\\ \notag
&\leq C(1+t)^{-1}\int_0^t E^{\frac32}_0(t')dt'.
\end{align}
According to the above estimate, we infer that
	\begin{align}\label{f1}
		\int_0^t\int_{S(t)}|\mathcal{F}(u\cdot\nabla u)\bar{\hat{u}}|d\xi dt'&\leq C\Big(\int_0^t \|u\|^2_{L^2}\|\nabla u\|_{L^2}dt'\Big)\Big(\int_{S(t)}d\xi\Big)^{\frac 12}\\ \notag
		&\leq C{{(1+t)^{-\frac12}}}\int_0^t
	(	E_0D^{\frac 12}_0)(t')dt',
	\end{align}
and
\begin{align}\label{f003}
&\int_{S(t)}\int_0^t|\mathcal {F}(h(\rho){\rm div}\Sigma(u)
)\bar{\hat{u}}|+|\mathcal{F}(h(\rho)\nabla\rho)\bar{\hat{u}}|+|\mathcal{F}(h(\rho){\rm div}\tau)\bar{\hat{u}}|dt'd\xi\\ \notag
&\leq C(1+t)^{-\frac 12}\int_0^tE_0\Big(D^{\frac 12}_0+D^{\frac 12}_1+\|\nabla \rho\|_{L^2}\Big)(t')dt'.
\end{align}
According to Lemma \ref{Lemma1}, one has
	\begin{align}
\int_0^t\int_{S(t)}\int_{B}\psi_\infty|\mathcal{F}({u\cdot\nabla g})|^2dRd\xi dt'
&\leq C\Big(\int_{S(t)}d\xi\Big)\Big(\int_0^t\|u\|^2_{L^2}\|\nabla\nabla_Rg\|^2_{L^2(\mathcal{L}^2)}dt'\Big) \\ \notag
&\leq C (1+t)^{-1}	\int_0^t
	E_0D_1
	dt'\leq C (1+t)^{-1},\\
	\int_0^t\int_{S(t)}\int_{B}\psi_\infty|\mathcal{F}(\nabla uRg)|^2 dRd\xi dt'
	&\leq C(1+t)^{-1}.\label{f3}
	\end{align}
	 Inserting \eqref{f00}-\eqref{f3} into \eqref{f0}, one can get
	\begin{align*}
	&\frac{d}{dt}\mathcal{E}_0(t)+\frac{\mu C_2}{2(1+t)}\|u\|^2_{H^2}+\frac{ \eta C_2}{2(1+t)}\|\rho\|^2_{H^2}+\frac{1}{2}\|\nabla_Rg\|^2_{H^2(\mathcal{L}^2)}\notag\\
	&\leq \frac{CC_2}{(1+t)^{2}}+\frac{CC_2}{(1+t)^2}\int_0^tE^{\frac32}_0(t')dt'+\frac{CC_2}{(1+t)^{\frac 3 2}} \int_0^t E_0
	\Big({D}_0^{
		\frac 12}+{D}_1^{
		\frac 12}(t')+
	\|\nabla\rho\|_{L^2}\Big)(t')
	dt'.
	\end{align*}
Therefore, taking $C_2$  large enough, we have
\begin{align*}
	&\frac{d}{dt}\Big((1+t)^{\frac32}\mathcal{E}_0(t)\Big)+\frac{\mu C_2(1+t)^{\frac 12}}{2}\|u\|^2_{H^2}+\frac{ \eta C_2(1+t)^{\frac 12}}{2}\|\rho\|^2_{H^2}+\frac{ (1+t)^{\frac 32}}{2}\|\nabla_Rg\|^2_{H^2(\mathcal{L}^2)}\notag\\
	&\leq C(1+t)^{-\frac 12}\Big(1+\int_0^tE^{\frac32}_0(t')dt'\Big)+C\int_0^t\mathcal{E}_0(t')\Big({D}_0^{
		\frac 12}(t')+{D}_1^{
		\frac 12}(t')+
	\|\nabla\rho\|_{L^2}\Big)dt',
\end{align*}
that is
\begin{align}\label{uu}
	(1+t)^{\frac32}\mathcal{E}_0(t)\leq C(1+t)^{\frac 12}+C(1+t)\int_0^t\mathcal{E}_0(t')	\Big({D}_0^{
		\frac 12}(t')+{D}_1^{
		\frac 12}(t')+
	\|\nabla\rho\|_{L^2}\Big)dt'.
\end{align}
	Define that $$X(t)=\sup\limits_{0\leq t'\leq t}(1+t')^{\frac 1 2}\mathcal{E}_0(t').$$  From \eqref{uu}, we obtain
	\begin{align*}
	X(t)	\leq C +C\int_{0}^{t}X(t')(1+t')^{-\frac 1 2}\Big(\mathcal{D}_0+\|\nabla\rho\|_{L^{2}}\Big)dt'.
	\end{align*}
	By virtue of  Gronwall's inequality and Lemma \ref{decayln} yields $X(t)\leq C$, which means that $$	\mathcal{E}_0(t)\leq C(1+t)^{-\frac 1 2}.$$
	Analogously,  from the above inequality and Remark \ref{re2}, we get
	\begin{align*}
		&\frac{d}{dt}\mathcal{E}_1(t)+\frac{C_2}{2(1+t)}\Big(\mu\|\nabla u\|^2_{H^1}+\eta\|\nabla\rho\|^2_{H^1}\Big)+\frac{1}{2}\|\nabla\nabla_Rg\|^2_{H^1(\mathcal{L}^2)}\\
		&\leq \frac{CC_2}{1+t}\int_{S(t)}|\xi|^2\Big(|\hat{u}(\xi)|^2+|\hat{\rho}(\xi)|^2\Big)d\xi\leq C(1+t)^{-{\frac 5 2}}.
	\end{align*}
Then, we have
\begin{align*}
	(1+t)^{\frac 52}\mathcal{E}_1(t)\leq C(1+t),
\end{align*}
which implies that
\begin{align*}
	\mathcal{E}_1(t)\leq C(1+t)^{-\frac 32}.
\end{align*}
	Hence, we arrive at \eqref{p1}.
	
	According to Remark \ref{re2} and the above estimate, we have
	\begin{align*}
		(1+t)^{\alpha_1}\mathcal{E}_1+\int_0^t(1+t')^{\alpha_1}\mathcal{D}_1dt
		&\leq C\mathcal{E}_1(0)+C\int_0^t(1+t')^{\alpha_1-1}\mathcal{E}_1dt'\\ \notag
		&\leq C\mathcal{E}_1(0)+C\int_0^t(1+t')^{\alpha_1-\frac 52}dt'\leq C.
	\end{align*}
	Therefore, the proof of Lemma \ref{decayp} is completed.
\end{proof}
The following  energy estimates with negative Besov spaces are crucial in our paper. To simplify the presentation, 
we give the following notation:  $$Y(t)=\sup\limits_{0\leq t'\leq t} (\|(\rho,u)\|_{\dot{B}^{-\sigma}_{2,\infty}}+\|g\|_{\dot{B}^{-\sigma}_{2,\infty}(\mathcal{L}^{2})})(t').$$
\begin{lemm}\label{leBe}
	Let $0<\alpha,\sigma\leq 1$ and $\sigma<2\alpha$. Assume that $(\rho_0,u_0,g_0)$ satisfy the conditions in Theorem \ref{opdecay'}. For any $t\in [0,+\infty)$, if
	\begin{align}\label{lea}
		\mathcal{E}_0(t)\leq C(1+t)^{-\alpha},~~~~\mathcal{E}_1(t)\leq C(1+t)^{-\alpha-1},
	\end{align}
	then we have
	\begin{align}\label{be1}
		(\rho,u)\in L^{\infty}(0,\infty;\dot{B}^{-\sigma}_{2,\infty}),~g\in L^{\infty}(0,\infty;\dot{B}^{-\sigma}_{2,\infty}(\mathcal{L}^2)).
	\end{align}	
\end{lemm}
\begin{proof} 
Applying $\dot{\Delta}_j$ to \eqref{eq1}, we get
	\begin{align}\label{j0}
		\left\{
		\begin{array}{ll}
			\dot{\Delta}_j\rho_t+{\rm div}\dot{\Delta}_j u=\dot{\Delta}_j F,  \\[1ex]
			\dot{\Delta}_j u_t-{\rm div}\Sigma(\dot{\Delta}_j u)+\nabla\dot{\Delta}_j \rho-{\rm div}\dot{\Delta}_j\tau=\dot{\Delta}_j H,  \\[1ex]
			\dot{\Delta}_j g_t+\mathcal{L}\dot{\Delta}_j g+{\rm div}\dot{\Delta}_j u-\nabla\dot{\Delta}_juR\nabla_R\mathcal{U}=\dot{\Delta}_j G, \\[1ex]
		\end{array}
		\right.
	\end{align}
	Multiplying both sides of \eqref{j0} by $2^{-2j\sigma}$ and taking $l^\infty$ norm yields
	\begin{align}\label{j1}
		&\frac d {dt} Y^2(t)+\mu\|\nabla u\|^2_{\dot{B}^{-\sigma}_{2,\infty}}+(\mu+\mu')\|{\rm div} u\|^2_{\dot{B}^{-\sigma}_{2,\infty}}+\|\nabla_Rg\|^2_{\dot{B}^{-\sigma}_{2,\infty}(\mathcal{L}^{2})}\\ \notag
		&
		\leq C\Big(\|F\|_{\dot{B}^{-\sigma}_{2,\infty}}\|\rho\|_{\dot{B}^{-\sigma}_{2,\infty}}+\|H\|_{\dot{B}^{-\sigma}_{2,\infty}}\|u\|_{\dot{B}^{-\sigma}_{2,\infty}}\Big)+ C\Big(\|\nabla uRg\|^2_{\dot{B}^{-\sigma}_{2,\infty}(\mathcal{L}^{2})}+\|u\nabla g\|^2_{\dot{B}^{-\sigma}_{2,\infty}(\mathcal{L}^{2})}\Big).
	\end{align}
	By using Lemmas \ref{intef}, \ref{decayp}, we arrive at the following two inequalities
	\begin{align}\label{j2}
		&\|u\cdot\nabla u\|_{\dot{B}^{-\sigma}_{2,\infty}}+\|h(\rho){\rm div}\tau\|_{\dot{B}^{-\sigma}_{2,\infty}}\\ \notag
		&\leq C\Big(\|u\|_{L^\frac {2} {\sigma}}\|\nabla u\|_{L^2}+\|\rho\|_{L^{\frac{2}{\sigma}}}\|\nabla\tau\|_{L^2}\Big)\\ \notag
		&\leq C\Big(\|u\|^{\sigma}_{L^2}\|\nabla u\|^{2-\sigma}_{L^2}+\|\rho\|^{\sigma}_{L^{2}}\|\nabla\rho\|^{1-\sigma}_{L^2}\|\nabla\nabla_Rg\|_{L^2(\mathcal{L}^2)}\Big) \\ \notag
			&\leq C\Big((1+t)^{-(1+\alpha-\frac \sigma 2)}+(1+t)^{-\frac12(1+\alpha-\sigma)}\|\nabla\nabla_Rg\|_{L^2(\mathcal{L}^2)}\Big) ,
	\end{align}
	and 
	\begin{align}\label{j3} 
		&\|\nabla uRg\|^2_{\dot{B}^{-\sigma}_{2,\infty}(\mathcal{L}^{2})}+\|u\nabla g\|^2_{\dot{B}^{-\sigma}_{2,\infty}(\mathcal{L}^{2})}\\ \notag
		&\leq C\Big(\|\nabla u\|^2_{L^2}\|g\|^2_{L^{\frac {2} {\sigma}}(\mathcal{L}^{2})}+\|\nabla g\|^2_{L^2(\mathcal{L}^{2})}\|u\|^2_{L^{\frac {2} {\sigma}}}\Big)\leq C(1+t)^{-(2+2\alpha-\sigma)}.
	\end{align}
From \eqref{j2},\eqref{j3}, we have
\begin{align}\label{j5}
	&\|F\|_{\dot{B}^{-\sigma}_{2,\infty}}\leq C(1+t)^{-(1+\alpha-\frac \sigma 2)},\\\label{j6}
	&\|H\|_{\dot{B}^{-\sigma}_{2,\infty}}\leq C\Big((1+t)^{-(1+\alpha-\frac \sigma 2)}+(1+t)^{-\frac12(1+\alpha-\sigma)}\|\nabla\nabla_Rg\|_{L^2(\mathcal{L}^2)}\Big).
\end{align}
	Plugging the estimates of \eqref{j3}-\eqref{j6} into \eqref{j1}, one has
	\begin{align*}
		\frac d {dt} Y^2(t)&\leq C\Big((1+t)^{-(1+\alpha-\frac \sigma 2)}+(1+t)^{-(2+2\alpha- \sigma )}\Big)Y(t)\\ \notag
		&~~~+C(1+t)^{-\frac12(1+\alpha-\sigma)}\|\nabla\nabla_Rg\|_{L^2(\mathcal{L}^2)}+C(1+t)^{-(2+2\alpha-\sigma)}.
	\end{align*}  
	By using \eqref{inte}, Lemma \ref{luoz} and the above estimate, we have
	\begin{align}\label{j7}
		&\int_0^T(1+t')^{-\frac12(1+\alpha-\sigma)}\|\nabla\nabla_Rg\|_{L^2(\mathcal{L}^2)}dt'\\ \notag
		&\leq C\Big(\int_0^t(1+t')^{-(1+2\alpha-\sigma)}dt'\Big)^{\frac12}\Big(\int_0^t(1+t')^{\alpha}\|\nabla\nabla_Rg\|^2_{L^2(\mathcal{L}^2)}dt'\Big)^{\frac12}\leq C.
	\end{align}
	 Integrating the above inequality with respect to $t'$ over $[0,t]$ and using \eqref{j7} yields that
	$$Y^2(t)\leq CY^2(0)+CY(t)+C,$$
	from which one can deduce that $Y(t)\leq C$.
	
	Then, we complete the proof of Lemma \ref{leBe}.
\end{proof}
The following lemma enables us obtain optimal decay rates. Moreover, after obtaining optimal decay rate of $(\rho,u)$ in $H^1$, we can use it to obtain optimal decay rate of $g$ in $L^2(\mathcal{L}^2)$ and the time weighted integrability result of $\mathcal{D}_1(t)$ with
$$\int_0^t(1+t')^{\alpha_2}\mathcal{D}_1(t')dt'\leq C,$$ which plays an important role in proving decay rate for the highest derivative of the solutions to \eqref{eq1}.
\begin{prop}\label{opdecay}
	Let  $(\rho_0,u_0,g_0)$ satisfy the conditions in Theorem \ref{opdecay'}. For any $s\in[0,1]$, there holds
	\begin{align*}
		&\|\nabla^s\rho\|_{L^2}+\|\nabla^su\|_{L^2}\leq C(1+t)^{-\frac {1+s} {2}},~~~\|\nabla^sg\|_{L^2(\mathcal{L}^2)}\leq C(1+t)^{-1}.
	\end{align*}
	Moreover, for any $\alpha_2\in[0,2)$, we have
	\begin{align}\label{integ}
		\int_0^t(1+t')^{\alpha_2}\mathcal{D}_1(t')dt'\leq C.
	\end{align}
\end{prop}
\begin{proof}
	The combination of Lemmas \ref{decayp}, \ref{leBe}  with $\sigma=\frac34$ and $\alpha=\frac 1 2$ implies that
	\begin{align}\label{b1}
		(\rho,u,g)\in L^{\infty}(0,\infty;\dot{B}^{-\frac 34}_{2,\infty})\times L^{\infty}(0,\infty;\dot{B}^{-\frac 34}_{2,\infty})\times L^{\infty}(0,\infty;\dot{B}^{-\frac 34}_{2,\infty}(\mathcal{L}^2)).
	\end{align}
Recall that
		\begin{align}\label{e00'}
		\frac{d}{dt}\mathcal{E}_0(t)+\frac{\mu C_2}{1+t}\|u\|^2_{H^2}+\frac{\eta C_2}{2(1+t)}\|\rho\|^2_{H^2}+\frac{1}{2}\|\nabla_Rg\|^2_{H^2(\mathcal{L}^2)}\leq \frac{CC_2}{1+t} \int_{S(t)}|\hat{u}(\xi)|^2+|\hat{\rho}(\xi)|^2d\xi.
	\end{align}
It follows from \eqref{b1} that
\begin{align}\label{f00'}
	\int_{S(t)}\Big(|\hat{u}|^2+|\hat{\rho}|^2\Big)
	d\xi
	&\leq C\sum_{j\leq \log_2[\frac {4} {3}C_2^{\frac 1 2 }{\frac{1}{\sqrt{1+t}}}]}\int_{\mathbb{R}^{2}} 2\varphi^2(2^{-j}\xi)\Big(|\hat{u}|^2+|\hat{\rho}|^2
	\Big)d\xi\\ \notag 
	&\leq C\sum_{j\leq \log_2[\frac {4} {3}C_2^{\frac 1 2 }\frac{1}{\sqrt{1+t}}]}\int_{\mathbb{R}^{2}}\|\dot{\Delta}_j u\|^2_{L^2}+\|\dot{\Delta}_j \rho\|^2_{L^2}
	d\xi\\ \notag
	&
	\leq C(1+t)^{-\frac34}\Big(\|u\|^2_{\dot{B}^{-\frac34}_{2,\infty}}+\|\rho\|^2_{\dot{B}^{-\frac34}_{2,\infty}}\Big).
\end{align}
Plugging the above estimate into \eqref{e00'}, one can deduce that
\begin{align}
		\frac{d}{dt}\mathcal{E}_0(t)+\frac{\mu C_2}{1+t}\|u\|^2_{H^2}+\frac{\eta C_2}{1+t}\|\rho\|^2_{H^1}+\frac{1}{2}\|\nabla_Rg\|^2_{H^2(\mathcal{L}^2)}\leq C(1+t)^{-\frac 74}.
\end{align}
According to Lemma \ref{decayp}, we conclude that
	\begin{align*}
		\mathcal{E}_0(t)\leq C(1+t)^{-\frac 34}.
	\end{align*}
On the other hand, by using \eqref{f00'}, we have the following estimate
	\begin{align}\label{e000}
	&\frac d {dt} \mathcal{E}_1(t)+\frac { C_2} {1+t}\Big(\mu\|\nabla u\|^2_{H^{1}}+\frac{\eta}{2}\|\nabla \rho\|^2_{H^1}\Big)
	+\frac{1}{2}\|\nabla \nabla_R g\|^2_{H^{1}(\mathcal{L}^{2})}  \\ \notag
	&\leq \frac {CC_2} {1+t}\int_{S(t)}|\xi|^2\Big(|\hat{u}(\xi)|^2+|\hat{\rho}(\xi)|^2\Big) d\xi\leq C(1+t)^{-2-\frac34},
\end{align}
from which one can deduce that
\begin{align*}
	\mathcal{E}_1(t)\leq C(1+t)^{-\frac 34-1}.
\end{align*}
	Lemma \ref{leBe} with $\sigma=1$ and $\alpha=\frac 34$  ensures that
	\begin{align*}
		(\rho,u,g)\in L^{\infty}(0,\infty;\dot{B}^{-1}_{2,\infty})\times L^{\infty}(0,\infty;\dot{B}^{-1}_{2,\infty})\times L^{\infty}(0,\infty;\dot{B}^{-1}_{2,\infty}(\mathcal{L}^2)).
	\end{align*}
Similar to \eqref{e00'} and \eqref{e000}, we have
\begin{align}\label{decaypug}
	\mathcal{E}_0(t)\leq C(1+t)^{-1},~~~~\mathcal{E}_1(t)\leq C(1+t)^{-2}.
\end{align}
On the other hand, Theorem \ref{global solution'} and Lemma \ref{Lemma1} yields
\begin{align*}
	\frac{d}{dt}\|g\|^2_{L^2(\mathcal{L}^2)}+c\|g\|^2_{L^2(\mathcal{L}^2)}\leq C\|\nabla u\|^2_{L^2},
\end{align*}
from which one can deduce that
\begin{align}\label{g1}
	\frac{d}{dt}\Big(e^{ct}\|g\|^2_{L^2(\mathcal{L}^2)}\Big)\leq Ce^{ct}\|\nabla u\|^2_{L^2}.
\end{align}
It follows from \eqref{decaypug} and \eqref{g1} that
\begin{align}
	\|g\|^2_{L^2(\mathcal{L}^2)}&\leq Ce^{-ct}\|g_0\|^2_{L^2(\mathcal{L}^2)}+\int_0^te^{-c(t-t')}\|\nabla u\|^2_{L^2}dt'\\ \notag
	&\leq Ce^{-ct}\|g_0\|^2_{L^2(\mathcal{L}^2)}+\int_0^te^{-c(t-t')}(1+t')^{-2}dt'\\ \notag
	&\leq C(1+t)^{-2}.
\end{align}
	According to Remark \ref{re2} and \eqref{decaypug}, for any $\alpha_2\in [0,2),$ we obtain
	\begin{align*}
		(1+t)^{\alpha_2}\mathcal{E}_1(t)+\int_0^t(1+t')^{\alpha_2}\mathcal{D}_1(t')dt'
		&\leq C\mathcal{E}_1(0)+C\int_0^t(1+t)^{\alpha_2-1}\mathcal{E}_1(t')dt'\\ \notag
		&\leq C\mathcal{E}_1(0)+\int_0^t(1+t')^{\alpha_2-3}\leq C\mathcal{E}_1(0)+C,
	\end{align*}
which plays an important role in proving the highest order decay rate.

Therefore, we complete the proof of Proposition \ref{opdecay}.
\end{proof}
Now we are devoting to studying   the optimal decay rate  for the highest derivative of $(\rho,u)$ in $L^2$ and the optimal decay rate of $\nabla g$ in $L^2(\mathcal{L}^2)$. Since the energy estimate  is unclosed,  we cannot obtain optimal decay rate for the highest derivative by using the same method in $H^1.$ To overcome this difficulty, we adopt a new method which flexibly combines the Fourier splitting
method and the time weighted energy estimate.  Therefore, we  obtain optimal decay rate of $\|\nabla^2(\rho,u)\|_{L^2}$. Finally, using the optimal rate of highest derivative for $(\rho,u)$ in $L^2$, we conclude that the optimal decay rate of $\nabla g$ in $L^2(\mathcal{L}^2)$.
\begin{prop}\label{highdecay}
	Under assumptions of Theorem \ref{opdecay'}, then  it holds 
	\begin{align}\label{h1}
		\|\nabla^2\rho\|_{L^2}+\|\nabla^2u\|_{L^2}\leq C(1+t)^{-\frac 32},~~~~~\|\nabla g\|_{L^2(\mathcal{L}^2)}\leq C(1+t)^{-{\frac 32}}.
	\end{align}
\end{prop}
\begin{proof}
For convenience, we firstly give the following notations:
\begin{align*}
	&\mathcal{E}_{2}(t)=\kappa(1+t)\Big(\|\nabla^2\rho\|^2_{L^2}+\|\nabla^2u\|^2_{L^2}+\|\nabla^2g\|^2_{L^2(\mathcal{L}^2)}\Big)+(\nabla u,\nabla^2\rho),\\ \notag 
	&\mathcal{D}_{2}(t)=\kappa(1+t)\Big(\mu\|\nabla^3u\|^2_{L^2}+(\mu+\mu')\|\nabla^2{\rm div}u\|^2_{L^2}+\|\nabla^2\nabla_{R}g\|^2_{L^2(\mathcal{L}^2)}\Big)+\frac1 4 \|\nabla^2\rho\|^2_{L^2},
\end{align*}
with $\kappa>0$ is sufficiently small.

We firstly deduce that
\begin{align}\label{34in1}
	\frac 1 2 \frac{d}{dt} \mathcal{E}_2(t)
	&=\kappa\Big(\|\nabla^2\rho\|^2_{L^2}+\|\nabla^2u\|^2_{L^2}+\|\nabla^2\nabla_Rg\|^2_{L^2(\mathcal{L}^2)}\Big) \\\notag
	&~~~+\kappa(1+t)\frac{d}{dt} {E}_2(t)+\frac{d}{dt}(\nabla u,\nabla^2\rho).
\end{align}
Making use of  \eqref{ineq14}-\eqref{ineq16''}, \eqref{decaypug} and Proposition  \ref{opdecay}, one can deduce that
\begin{align}\label{h0}
	\kappa(1+t)(J''_1+J''_2)
	&
	\leq C\kappa(1+t)\|\nabla u\|^{\frac 12}_{L^2}\|\nabla^2\rho\|^{\frac 12}_{L^2}\|\nabla^3 u\|^{\frac 12}_{L^2}\|\nabla^2\rho\|^{\frac 32}_{L^2}\\\notag
	&\leq C\kappa \|\nabla^3 u\|^{\frac 12}_{L^2}\|\nabla^2\rho\|^{\frac 32}_{L^2}\leq C{\kappa^4(1+t)}\|\nabla^3 u\|^{2}_{L^2}+\frac{1}{100}\|\nabla^2\rho\|^{2}_{L^2},
\end{align}
and
\begin{align}\label{h0'}
		\kappa(1+t)J''_3&\leq C\kappa(1+t)\|\rho\|^{\frac 12}_{L^2}\|\nabla^2\rho\|^{\frac 12}_{L^2}\|\nabla^3u\|_{L^2}\|\nabla^2\rho\|_{L^2}\\\notag
		&\leq C\kappa (1+t)^{\frac 14}\|\nabla^2\rho\|_{L^2}\|\nabla^3u\|_{L^2}\leq C\kappa^2(1+t)\|\nabla^3u\|^2_{L^2}+\frac{1}{100}\|\nabla^2\rho\|^{2}_{L^2},
\end{align} 
where we have used that $$\|\nabla u\|^{\frac 12}_{L^2}\|\nabla^2\rho\|^{\frac 12}_{L^2}\leq C\mathcal{E}^{\frac 12}_1\leq C(1+t)^{-1},~\|\rho\|^{\frac 12}_{L^2}\|\nabla^2\rho\|^{\frac 12}_{L^2}\leq C\mathcal{E}^{\frac 14}_0\mathcal{E}^{\frac 14}_1\leq C(1+t)^{-\frac 34}.$$

In order to obtain optimal decay rate of solutions in $\dot{H}^2$ with large data, we now give another the estimates of $J''_4, J''_5$. Using Lemma \ref{intef} with $\|\nabla^2u\|_{L^4}\leq C\|\nabla u\|^{\frac14}_{L^2}\|\nabla^3 u\|^{\frac34}_{L^2}$ and Proposition \ref{opdecay}, we have
\begin{align*}
	J''_4+J''_5\leq C\|\nabla\rho\|^{\frac 12}_{L^2}\|\nabla u\|^{\frac14}_{L^2}\|\nabla^3 u\|^{\frac34}_{L^2}\|\nabla^2\rho\|^{\frac 32}_{L^2}\leq C(1+t)^{-\frac34}\|\nabla^3 u\|^{\frac34}_{L^2}\|\nabla^2\rho\|^{\frac 32}_{L^2}.
\end{align*}
Hence, using the above estimate yields 
\begin{align}\label{h1'}
	\kappa(1+t)(J''_4+J''_5)\leq C\kappa(1+t)^{\frac 14}\|\nabla^3 u\|^{\frac34}_{L^2}\|\nabla^2\rho\|^{\frac 32}_{L^2}\leq C\kappa^2(1+t)\|\nabla^3u\|^2_{L^2}+C\kappa^{\frac25}\|\nabla^2\rho\|^2_{L^2},
\end{align}
and
\begin{align}
	\kappa(1+t)J''_6&\leq C\kappa(1+t)^{\frac 34}\|\nabla^3u\|^2_{L^2}\leq \frac{\kappa(1+t)} {100}\|\nabla^3u\|^2_{L^2}.
\end{align}
Similar to  the estimate of $\kappa(1+t)J^{''}_3$, we infer that
\begin{align}
		\kappa(1+t)(J''_7+J^{''}_8)\leq C\kappa^2(1+t)\Big(\|\nabla^3u\|^2_{L^2}+\|\nabla^2\nabla_Rg\|^2_{L^2(\mathcal{L}^2)}\Big)+\frac{1}{100}\|\nabla^2\rho\|^{2}_{L^2}.
\end{align}
For the estimates of $	\kappa(1+t)(\sum\limits_{i=9}^{12}J''_i)$, one has
\begin{align}
	\kappa(1+t)(J''_9+J''_{10}+J''_{11})
	&\leq C\kappa(1+t)^{\frac 12}\|\nabla^3u\|_{L^2}\Big(\|\nabla^2u\|_{L^2}+\|\nabla^2\rho\|_{L^2}\Big)\\\notag
	&\leq\Big(C\kappa+\frac 1{100}\Big)\mathcal{D}_2+C\|\nabla^2u\|^2_{L^2},\\ 
	\kappa(1+t)J''_{12}\leq C\kappa(1+t)^{\frac 34}
	&\|\nabla^3u\|^{\frac 32}_{L^2}\|\nabla^2u\|^{\frac 12}_{L^2}\leq C\kappa^{\frac 13}\mathcal{D}_2+C\|\nabla^2u\|^2_{L^2}.
	\label{o} 
\end{align}
Making use of \eqref{decaypug}, we get $$\|\nabla\rho\|^{\frac 12}_{L^2}+\|\nabla^2\rho\|^{\frac 14}_{L^2}\|\nabla g\|^{\frac 14}_{L^2(\mathcal{L}^2)}\leq C\mathcal{E}_1^{\frac 14}\leq C(1+t)^{-\frac 12}.$$ Consequently, using H{\"o}lder's inequality, we arrive at 
\begin{align}\label{o0}
		\kappa(1+t)J''_{13}
	&\leq C\kappa(1+t)\|\nabla\rho\|^{\frac 12}_{L^2}\|\nabla^2\rho\|^{\frac 12}_{L^2}\|\nabla g\|^{\frac 14}_{L^2(\mathcal{L}^2)}\|\nabla \nabla_Rg\|^{\frac 14}_{L^2(\mathcal{L}^2)}\|\nabla^2\nabla_Rg\|^{\frac 12}_{L^2(\mathcal{L}^2)}\|\nabla^3u\|_{L^2}\\ \notag
	& \leq C\kappa(1+t)^{\frac 12}\|\nabla^2\rho\|^{\frac 12}_{L^2}\|\nabla g\|^{\frac 14}_{L^2(\mathcal{L}^2)}\|\nabla \nabla_Rg\|^{\frac 14}_{L^2(\mathcal{L}^2)}\|\nabla^2\nabla_Rg\|^{\frac 12}_{L^2(\mathcal{L}^2)}\|\nabla^3u\|_{L^2}\\ \notag
	& \leq \frac{\kappa(1+t)}{100}\|\nabla^3u\|^2_{L^2}+C\kappa \|\nabla^2\rho\|_{L^2}\|\nabla g\|^{\frac 12}_{L^2(\mathcal{L}^2)}\|\nabla \nabla_Rg\|^{\frac 12}_{L^2(\mathcal{L}^2)}\|\nabla^2\nabla_Rg\|^{\frac 12}_{L^2(\mathcal{L}^2)}\\ \notag
	& \leq \frac{\kappa(1+t)}{100}\|\nabla^3u\|^2_{L^2}+\frac{\kappa(1+t)}{100}\|\nabla^2\nabla_Rg\|^2_{L^2(\mathcal{L}^2)}\\ \notag
	&~~~+C\kappa \|\nabla^2\rho\|^2_{L^2}\|\nabla g\|_{L^2(\mathcal{L}^2)}\|\nabla \nabla_Rg\|_{L^2(\mathcal{L}^2)}\\ \notag
	& \leq \frac{\kappa(1+t)}{100}\|\nabla^3u\|^2_{L^2}+\frac{\kappa(1+t)}{100}\|\nabla^2\nabla_Rg\|^2_{L^2(\mathcal{L}^2)}+\frac{1}{100}\|\nabla^2\rho\|^2_{L^2}\\ \notag
	&~~~+C\|\nabla^2\rho\|^2_{L^2}\|\nabla g\|^2_{L^2(\mathcal{L}^2)}\|\nabla \nabla_Rg\|^2_{L^2(\mathcal{L}^2)}\\ \notag
	&\leq \frac{1}{100}\mathcal{D}_2(t)+C(1+t)^{-4}\|\nabla \nabla_Rg\|^{2}_{L^2(\mathcal{L}^2)}.
\end{align}
where we have used the inequality $$\|\nabla^2\rho\|^{2}_{L^2}\|\nabla g\|^{2}_{L^2(\mathcal{L}^2)}\|\nabla \nabla_Rg\|^{2}_{L^2(\mathcal{L}^2)}\leq C(1+t)^{-4}\|\nabla \nabla_Rg\|^{2}_{L^2(\mathcal{L}^2)}.$$
According to  \eqref{decaypug}, we deduce rhat $$\|\nabla g\|^{\frac 12}_{L^2(\mathcal{L}^2)}\|\nabla^2u\|^{\frac 12}_{L^2}\leq C\mathcal{E}_1^{\frac 12}\leq C(1+t)^{-1},~~\|g\|_{L^{\infty}(\mathcal{L}^2)}\leq C\mathcal{E}_0^{\frac 14}\mathcal{E}_1^{\frac 14}\leq C(1+t)^{-\frac 34},$$  
from which one can have
\begin{align}
	\kappa(1+t)(J''_{14}+J''_{15})&\leq C\kappa\Big(\|\nabla^3u\|^{\frac12}_{L^2}\|\nabla^2\nabla_Rg\|^{\frac32}_{L^2(\mathcal{L}^2)}\Big)\\ \notag
	&\leq \frac{\kappa(1+t)}{100}\Big(\|\nabla^3u\|^{2}_{L^2}+\|\nabla^2\nabla_Rg\|^{2}_{L^2(\mathcal{L}^2)} \Big),
\end{align}
and
\begin{align}
		\kappa(1+t)J''_{16}\leq C\kappa(1+t)^{\frac 14}\|\nabla^3u\|_{L^2}\|\nabla^2\nabla_Rg\|_{L^2(\mathcal{L}^2)}\leq \frac{\kappa(1+t)}{100}\Big(\|\nabla^3u\|^{2}_{L^2}+\|\nabla^2\nabla_Rg\|^{2}_{L^2(\mathcal{L}^2)} \Big).
\end{align}
For the estimates of $	\kappa(1+t)(\sum\limits_{i=17}^{19}J''_i)$, one can deduce that
\begin{align}\label{h89}
	\kappa(1+t)(\sum\limits_{i=17}^{19}J''_i)&\leq C\kappa\|\nabla^3u\|^{\frac 12}_{L^2}\|\nabla^2\nabla_Rg\|^{\frac 32}_{L^2(\mathcal{L}^2)}\\ \notag
	&\leq \frac{\kappa(1+t)}{100}\Big(\|\nabla^3u\|^{2}_{L^2}+\|\nabla^2\nabla_Rg\|^{2}_{L^2(\mathcal{L}^2)} \Big).
\end{align}
Combining \eqref{h0}-\eqref{h89}, we obtain 
\begin{align*}
	\kappa(1+t)\frac{d}{dt}E_2(t)+\kappa(1+t)D_2\leq
	&\Big(C\kappa+C\kappa^{\frac13}+\frac14\Big)\mathcal{D}_2+C\|\nabla^2u\|^2_{L^2}+C(1+t)^{-4}\|\nabla \nabla_Rg\|^{2}_{L^2(\mathcal{L}^2)}.
\end{align*}
According to \eqref{b7}, we have
\begin{align}\label{o7}
	(\nabla^3u,F)
	&\leq CE^{\frac 12}_0\|\nabla^3u\|_{L^2}(\|\nabla^2u\|^{\frac 14}_{L^2}\|\nabla^2\rho\|^{\frac34}_{L^2}+\|\|\nabla^2u\|^{\frac 34}_{L^2}\|\nabla^2\rho\|^{\frac14}_{L^2})\\ \notag
	&\leq C{\kappa^2(1+t)}\|\nabla^3u\|^2_{L^2}+\frac{C\kappa}{4}\|\nabla^2\rho\|^2_{L^2}+C\|\nabla^2u\|^2_{L^2}\leq C\kappa\mathcal{D}_2+C\|\nabla^2u\|^2_{L^2}.
\end{align}
 According to \eqref{b7'} and \eqref{b7''}, for $t$ large enough, we deduce that
\begin{align}
	\|\nabla^2\rho\|_{L^2}\|\rho\|_{L^{\infty}}\Big(\|\nabla^3 u\|_{L^2}+\|\nabla^2\rho\|_{L^2}\Big)
	&\leq C\kappa^2(1+t)\|\nabla^3 u\|^2_{L^2}+C\|\rho\|_{L^{\infty}}\|\nabla^2\rho\|^2_{L^2}\\ \notag
	&\leq C\kappa\mathcal{D}_2+C\kappa\|\nabla^2\rho\|^2_{L^2},\\ 
	\|\nabla^2\rho\|_{L^2}(\|\rho\|_{L^{\infty}}\|\nabla^2 \tau\|_{L^2}+\|u\|_{L^{\infty}}\|\nabla^2u\|_{L^2})
	&\leq \frac18\|\nabla^2\rho\|^2_{L^2}+C\kappa\mathcal{D}_2+C\|\nabla^2u\|^2_{L^2},\label{o6}
\end{align}
and
\begin{align}\label{o5}
	(\nabla{\rm div}\tau,\nabla^2\rho)+(\nabla{\rm div}\Sigma(u),\nabla^2\rho)
	&\leq \|\nabla^2\rho\|_{L^2}\Big(\|\nabla^2\tau\|_{L^2}+\|\nabla^3u\|_{L^2}\Big)\leq \frac18 \|\nabla^2\rho\|^2_{L^2}+ C\kappa\mathcal{D}_2.
\end{align}
According to \eqref{o7}-\eqref{o5}, we have
\begin{align}\label{ineq90}
	\frac{d}{dt}(\nabla u,\nabla^2\rho)+\frac12 \|\nabla^2\rho\|^2_{L^2}\leq C\kappa\mathcal{D}_2+C\|\nabla^2u\|^2_{L^2}.
\end{align}
Making use of Lemma \ref{intef} and Proposition \ref{decayp}, we obtain
\begin{align*}
	\|\nabla^2 u\|^2_{L^2}\leq C\|u\|^{\frac 2 3}_{L^2}\|\nabla^3u\|^{\frac 4 3}_{L^2}\leq C(1+t)^{-2}\|u\|^2_{L^2}+\frac{\kappa(1+t)}{4}\|\nabla^3u\|^2_{L^2}.
\end{align*}
Therefore, we get
\begin{align}\label{h11}
	\frac{d}{dt}\mathcal{E}_2(t)+2\mathcal{D}_2(t)&\leq
	\kappa\Big(\|\nabla^2\rho\|^2_{L^2}+\|\nabla^2u\|^2_{L^2}+\|\nabla^2\nabla_Rg\|^2_{L^2(\mathcal{L}^2)}\Big)\\ \notag
	&~~~ +\Big(C\kappa+C\kappa^{\frac13}+\frac12\Big)\mathcal{D}_2+C(1+t)^{-3}+C(1+t)^{-4}\|\nabla \nabla_Rg\|^{2}_{L^2(\mathcal{L}^2)},
\end{align}
from which one has
\begin{align*}
	\frac{d}{dt} \mathcal{E}_2(t)+\mathcal{D}_2(t) &\leq\kappa\Big(\|\nabla^2\rho\|^2_{L^2}+\|\nabla^2u\|^2_{L^2}+\|\nabla^2\nabla_Rg\|^2_{L^2(\mathcal{L}^2)}\Big)\\ \notag
	&~~~+ C(1+t)^{-3}+C(1+t)^{-4}\|\nabla \nabla_Rg\|^{2}_{L^2(\mathcal{L}^2)}.
\end{align*}
By using the Fourier splitting method and using Proposition \ref{opdecay}, we have
\begin{align}\label{hh}
	&(1+t)^3\mathcal{E}_2+\int_0^t\kappa (1+t')^3\Big(\|\nabla^2(\rho,u)\|^2_{L^2}+\kappa(1+t')\|\nabla^2\nabla_Rg\|^2_{L^2(\mathcal{L}^2)}\Big)dt'\\ \notag
	&\leq C(1+t)+C\int_0^t(1+t')^{-1}\|\nabla \nabla_Rg\|^{2}_{L^2(\mathcal{L}^2)}dt'+C\int_0^t(1+t')^2(\nabla u,\nabla^2\rho)dt'\\ \notag 
	&\leq C(1+t)+\frac{\kappa}{2}\int_0^t(1+t')^3\|\nabla^2\rho\|^2_{L^2}dt'+C\int_0^t(1+t')\|\nabla u\|^2_{L^2}dt'\\\notag 
	&\leq C(1+t) +\frac{\kappa}{2}\int_0^t(1+t')^3\|\nabla^2\rho\|^2_{L^2}dt',
\end{align}
from which one can deduce that
\begin{align*}
	(1+t)^3\mathcal{E}_2(t)+\int_0^t\kappa (1+t')\Big(\|\nabla^2(\rho,u)\|^2_{L^2}+\kappa(1+t')^3\|\nabla^2\nabla_Rg\|^2_{L^2(\mathcal{L}^2)}\Big)dt'\leq C(1+t).
\end{align*}
That is
\begin{align}\label{H}
	\|\nabla^2(\rho,u)\|^2_{L^2}+\|\nabla^2g\|^2_{L^2(\mathcal{L}^2)}\leq C(1+t)^{-3}.
\end{align}
Next we give the decay rate of $\|\nabla g\|^2_{L^2(\mathcal{L}^2)}.$ It follows  from  Lemma \ref{lemma2} that
\begin{align*}
	\frac{d}{dt}\|\nabla g\|^2_{L^2(\mathcal{L}^2)}+\|\nabla\nabla_R g\|^2_{L^2(\mathcal{L}^2)}&\leq C\|\nabla^2u\|^2_{L^2}+\|\nabla u\|^2_{L^4}\|\nabla g\|^2_{L^4(\mathcal{L}^2)}\\ \notag
	&~~~+\|g\|^2_{L^{\infty}(\mathcal{L}^2)}\|\nabla^2u\|^2_{L^2}+\|u\|^2_{L^\infty}\|\nabla g\|^2_{L^2(\mathcal{L}^2)}.
\end{align*} 
By applying  Lemmas \ref{intef}, \ref{Lemma1} yields
\begin{align*}
	\frac{d}{dt}\Big(e^{ct}\|\nabla g\|^2_{L^2(\mathcal{L}^2)}\Big)\leq Ce^{ct}\Big(\|\nabla^2u\|^2_{L^2}+E_1E_2+E^{\frac 12}_0E^{\frac 32}_2+E^{\frac 12}_0E^{\frac 32}_2\Big).
\end{align*} 
It can deduce from \eqref{H} and Proposition \ref{opdecay} that
\begin{align*}
	\|\nabla g\|^2_{L^2(\mathcal{L}^2)}\leq Ce^{-ct}\|\nabla g_0\|^2_{L^2(\mathcal{L}^2)}+C\int_0^te^{-c(t-t')}(1+t')^{-3}dt',
\end{align*}
from which one can arrive at
\begin{align*}
	\|\nabla g\|^2_{L^2(\mathcal{L}^2)}\leq C(1+t)^{-3}.
\end{align*}
Therefore, we complete the proof of  Proposition \ref{highdecay}.
\end{proof}
We now come to the proof of optimal decay estimates of solutions for \eqref{eq1}.\\
\textbf{Proof of Theorem \ref{opdecay'}:}
\begin{proof}
	Combining Propositions \ref{opdecay} and  \ref{highdecay}, we finish the proof of Theorem \ref{opdecay'}.
\end{proof}
\smallskip
\noindent\textbf{Acknowledgments} This work was partially supported by the National Natural Science Foundation of China (No. 12171493), the China Postdoctoral Science Foundation (No. 2022TQ0077 and No. 2023M730699) and Shanghai Post-doctoral Excellence Program (No. 2022062).

\noindent\textbf{Data Availability.}
The data that support the findings of this study are available on citation. The data that support the findings of this study are also available
from the corresponding author upon reasonable request.

	\phantomsection
	\addcontentsline{toc}{section}{\refname}
	\bibliographystyle{abbrv} 
	\bibliography{Feneref00.1}
\end{document}